\documentclass[cupthm,crop,info]{CUP-JNL-ETS}%

\usepackage{graphicx}
\usepackage{multicol,multirow}
\usepackage{amsmath,amssymb,amsfonts}
\usepackage{mathrsfs}
\usepackage{rotating}
\usepackage{appendix}
\usepackage[numbers]{natbib}
\usepackage{amsthm}
\usepackage{lipsum}
\usepackage{mathtools}

\usepackage{tikz}
\usetikzlibrary[calc, math, decorations.pathreplacing]

\theoremstyle{cupplain}
\newtheorem{theorem}{Theorem}[section]
\newtheorem{lemma}[theorem]{Lemma}
\newtheorem{corollary}[theorem]{Corollary}
\theoremstyle{cupdefinition}
\newtheorem{definition}{Definition}[section]
\theoremstyle{cupremark}
\newtheorem{remark}[theorem]{Remark}

\theoremstyle{cupproof}
\newtheorem{proof}{Proof}

\numberwithin{equation}{section}

\usepackage{xcolor}
\usepackage[utf8]{inputenc}
\usepackage{hyperref}
\usepackage{enumitem}
\usepackage{soul}
\usepackage{cancel}

\long\def\MSC#1\EndMSC{\def\arg{#1}\ifx\arg\empty\relax\else {\narrower\noindent%
{2010 Mathematics Subject Classification}: #1\\} \fi}
\long\def\PACS#1\EndPACS{\def\arg{#1}\ifx\arg\empty\relax\else
     {\narrower\noindent%
{PACS numbers}: #1}\fi}
\long\def\KEY#1\EndKEY{\def\arg{#1}\ifx\arg\empty\relax\else
	{\narrower\noindent%
Keywords: #1\\}\fi}

\newcommand{\Z}{\mathbb{Z}}
\newcommand{\R}{\mathbb{R}}
\newcommand{\C}{\mathbb{C}}

\newcommand{\bj}{\mathbf{j}}
\newcommand{\bk}{\mathbf{k}}

\newcommand{\W}{\mathscr{P}}

\renewcommand{\e}{\mathrm{e}}

\def\bb1{{\rm{1}\hspace{-3pt}\mathbf{l}}}

\renewcommand{\d}{\;\mathrm{d}} 

\renewcommand{\epsilon}{\varepsilon}
\renewcommand{\phi}{\varphi}

\artid{}
\jyear{}
\jvol{}
\jissue{}
\jdoi{}
\begin{document}

\begin{Frontmatter}

\title[Beta-expansions and sharp iteration asymptotics]{Beta-expansions induced by general Parry polynomials and sharp iteration asymptotics of their transfer operators}







\author{\gname{Horia D.} \sname{Cornean*}}
\author{\gname{Ira W.} \sname{Herbst**}}
\author{\gname{Kasper S.} \sname{Sørensen*}}

\address{\orgdiv{*Department of Mathematical Sciences}, \orgname{Aalborg University}, \orgaddress{\city{Aalborg},  \country{Denmark}} 
\orgdiv{**Department of Mathematics}, \orgname{University of Virginia}, \orgaddress{\city{Charlottesville},  \country{USA}} \\
(\email{cornean@math.aau.dk},\email{iwh@virginia.edu},\email{kasper@math.aau.dk})}

\Received{\textup{XX} MONTH \textup{XX}}
\Accepted[and accepted in revised form]{\textup{XX} MONTH \textup{XXX}}

\maketitle
\thispagestyle{empty}
\tableofcontents

\begin{abstract}
This paper generalizes the results in \cite{CHM,CS, HMS, HMS2} to any Parry polynomials $p_n(z)$ \cite{Pa} of finite degree $n\geq 2$. Let $\beta>1$ be the unique positive root of $p_n$ and consider the transfer operator $\mathscr{P}$ induced by the map $[0,1]\ni x\mapsto \beta x-\lfloor \beta x\rfloor \in [0,1)$. Suppose all other roots of $p_n$ have magnitude less than one. Then for any smooth function $F$ on $[0,1]$ we prove two term asymptotics in  $L^\infty$ of the form 
\begin{equation*}
    \big (\mathscr{P}^r(F)\big )(\cdot)= u(\cdot) \int_0^1 F(x)\d x +\beta^{-r}v(\cdot) \big (F(1)-F(0)\big )+o(\beta^{-r}),\quad r\geq 1,
\end{equation*}
where $u$ is piecewise constant and $\mathscr{P}u=u$, while $v$ is piecewise equal to a polynomial of degree one and $\mathscr{P}v= \beta^{-1}v$. 

When $\W$ is restricted to functions of bounded variation $BV([0,1])$, it is known  \cite{HK', HK, K, LY} that $\vert z \vert\leq \beta^{-1}$ contains its essential spectrum. Thus the contribution coming from $v$ in the above estimate is not guaranteed if $F$ is not smooth enough. 

Furthermore, we show that when $\W$ is restricted to $BV([0,1])$ then any $\vert z \vert<\beta^{-1}$ is an eigenvalue with infinite multiplicity. Moreover, we construct a meromorphic extension of $(\W-z)^{-1}$ to the domain $\vert z \vert>\beta^{-1}$, where the finite rank poles are numbers $\zeta$ with $\beta^{-1}<\vert \zeta \vert\leq 1$ such that $\beta\zeta$ is a root of $p_n(z)$.  If such poles $\zeta$ other than $1$ exist, then we construct eigenvectors of $\W^*$ in $BV^*$ obeying $\W^*\hat{U}=\zeta \hat{U}$, { where $\hat{U}(F)$} is given by series involving the values of $F$ taken in some specially chosen, finer and finer partitions of $[0,1]$. When $F$ is either a polynomial or a characteristic function like $\chi_{[0,x]}$,  { then $\hat{U}(F)$} can be explicitly computed. The case where $F$ is a characteristic function has been treated in \cite{Su} with a different approach,  but the general construction of $\hat{U}$ is first given here. 
\end{abstract}


\end{Frontmatter}

\section{Introduction}
\subsection{The setting}

 Let $n \in\{2,3,\ldots\}$, $a_i \in \mathbb{N}_0 = \{0,1,2,\ldots\}$ for $i \in \{0,1,\ldots,n-1\}$ and $a_0 a_{n-1}\geq 1$. 
 Consider the function 
\begin{align}\label{dec2}
    f(x) = 1 - \frac{a_0}{x} - \frac{a_1}{x^2} - \cdots - \frac{a_{n-1}}{x^{n}}.  
\end{align}
We have $f'(x) > 0$ for all $x > 0$, $f(a_0) < 0$ and $\lim_{x \to \infty} f(x) = 1$. This implies that the equation  
\begin{align}\label{eq:fundeqbeta}
    1 = \frac{a_0}{\beta} + \frac{a_1}{\beta^2} + \cdots + \frac{a_{n-1}}{\beta^n},\quad \beta>0,
\end{align}
has a unique solution.  In this context, we see that $\beta>1$ is the unique positive root of the degree $n$ polynomial with integer coefficients 
\begin{equation}\label{dec2'}
p(z)=z^n-a_0z^{n-1}-\cdots -a_{n-1}.
\end{equation}

Let $T_\beta(x) = x \beta - \lfloor x \beta \rfloor$, and assume that $\beta$ and the corresponding coefficients $\{a_j\}_{j=0}^{n-1}$ also satisfy the following constraints: 
\begin{align}\label{eq:a0a1akassump}
    a_0 = \lfloor \beta \rfloor, \quad a_1 = \lfloor \beta T_\beta(1) \rfloor, \quad \ldots, \quad a_{n-1} = \lfloor \beta T^{n-1}_\beta(1) \rfloor,\quad T^n_\beta(1)=0 .
\end{align}

We call such $a_j$'s a finite Parry sequence; Parry (see \cite{Pa}) found the necessary and sufficient conditions which the coefficients must satisfy in order to also obey \eqref{eq:fundeqbeta} and \eqref{eq:a0a1akassump}: denoting by $\succ$ the lexicographic ``greater than" order between two sequences, then these conditions are: 
\begin{align}\label{Parrycon}
    [a_0,a_1,\ldots,a_{n-1},0,\ldots ]\succ [a_{j},a_{j+1},\ldots,a_{n-1},0,\ldots ]\quad \text{for all}\quad {1\leq j\leq n-1}.
\end{align}

\begin{definition}\label{def1}
We say that a polynomial $p(z)$ as in \eqref{dec2'} is a Parry polynomial if the coefficients $\{a_j\}_{j=0}^{n-1}$ form a finite Parry sequence. We say that $p(z)$ is a Parry-Pisot polynomial if it is a Parry polynomial and the other $n-1$ roots of $p(z)$ besides $\beta$ belong to the open unit disk. 
\end{definition}
\begin{definition}
    We say that $\beta>1$ is a Pisot number if there exists a polynomial $p(z)=z^n+\sum_{k=1}^{n} c_k z^{n-k}$ with all $c_k\in \Z$,  such that $p(\beta)=0$ and all other roots of $p$ lie in the open unit disk. 
\end{definition}

\begin{remark}Note that a Parry polynomial may fail to be a Parry-Pisot polynomial although its corresponding $\beta$ is a Pisot number. This happens for example for the (reducible over integers) polynomial 
\begin{equation}\label{reduc}
z^5-z^4-1=(z^2-z+1)(z^3-z-1)
\end{equation}
where $n=5$, $a_0=a_4=1$ and $a_1=a_2=a_3=0$; this is a finite Parry sequence, see \eqref{dec6}. Here $\beta$ is the positive root of $z^3-z-1$, the so-called ``plastic constant", the smallest possible Pisot number. But $p(z)$ in \eqref{reduc} is not Parry-Pisot because two of its roots are $e^{\pm i\pi/3}$ and have length one and the other two roots of $z^3-z-1$ have length less than $1$. 
\end{remark}

For any $\beta>1$ coming from a finite Parry sequence we introduce the transfer (or Perron-Frobenius) operator $\mathscr{P}:L^p([0,1])\to L^p([0,1])$ with $1\leq p\leq \infty$ given by 
\begin{align}\label{eq:mathcalP}
    (\mathscr{P} f)(x) = \beta^{-1}\sum_{k=0}^{a_0} f\bigg( \frac{x+k}{\beta} \bigg) ,\quad 0\leq x\leq 1,
\end{align}
 Here we always assume $f(x) = 0$ if $x\notin [0,1]$.

Let us also introduce the composition (Koopman) operator $K:L^p([0,1])\to L^p([0,1])$ acting as
\begin{align}\label{koop}
    (Kf)(x) = f\big(T_\beta(x)\big)=f\big(\beta x - \lfloor \beta x \rfloor\big).
\end{align}
There is an important duality relation between these two operators. For all $g\in L^1([0,1])$ and $f\in L^\infty([0,1])$:
\begin{align}\label{dec12}
    \int_{[0,1]} g(x)\, (Kf)(x)\d x = \int_{[0,1]} (\mathscr{P}g)(x)\, f(x)\d x.
\end{align}

\subsection{More on the necessary and sufficient conditions for having a finite Parry sequence} 
Following \cite{Pa} we will give here a few more details on the necessary and sufficient  conditions which the coefficients must obey when they satisfy \eqref{eq:a0a1akassump} and \eqref{Parrycon}. 

The first equality in \eqref{eq:a0a1akassump} demands $a_0 < \beta < a_0+1$ (where $\beta > a_0$ holds due to \eqref{eq:fundeqbeta}). Furthermore, by $T_\beta : [0,1] \rightarrow [0,1)$ together with  \eqref{eq:fundeqbeta} and \eqref{eq:a0a1akassump}, we have for all $1\leq k\leq n-1$:
\begin{equation}\label{dec1}
\begin{aligned}
b_0&\coloneqq 1,\\
    b_1 &\coloneqq \frac{a_1}{\beta} + \cdots + \frac{a_{n-1}}{\beta^{n-1}} =T_\beta(1)= \beta - a_0  \in (0,1), \\
    b_2 &\coloneqq \frac{a_2}{\beta} + \cdots + \frac{a_{n-1}}{\beta^{n-2}} = T^2_\beta(1)=\beta^2 - \beta a_0 - a_1 \in (0,1), \\
    & \:\:\: \vdots \\
    b_{k} &\coloneqq \frac{a_{k}}{\beta} + \cdots + \frac{a_{n-1}}{\beta^{n-k}} = T^k_\beta(1)=\beta^{k} - \beta^{k-1} a_0 - \beta^{k-2} a_1 - \cdots - a_{k-1} \in (0,1).
\end{aligned}
\end{equation}

We gather a few important facts implied by \eqref{eq:a0a1akassump}.
\begin{lemma}\label{lemmaParry} 
    We have $b_j\neq b_k$ if $j\neq k$. 
\end{lemma}
\begin{proof}
If $j>k$ and $T^j(1) =T^k(1)$ then $T^{n-j}$ acting on $T^j(1)$ gives zero but acting on $T^k(1)$ does not, which leads to a contradiction. 
\end{proof}

\begin{lemma}\label{a0islarger}  Let $a_0\geq \max\{a_1,\ldots,a_{n-1}\}$ be non-negative integers with $a_0 a_{n-1}\geq 1$ but which do not necessarily obey \eqref{Parrycon}. 
Let $\beta $ be given by \eqref{eq:fundeqbeta}. 
Then $a_0<\beta<a_0+1$.     
 Moreover, if $2\leq n\leq 3$, then $a_0\geq \max\{a_1,\ldots,a_{n-1}\}$ is  also sufficient for the finite sequence to be Parry. 
\end{lemma}
\begin{proof}
    Indeed, from \eqref{dec2} we get 
\begin{equation*}
    f(a_0)<0,\quad f(a_0+1)> 1-\sum_{j\geq 1}\frac{a_0}{(a_0+1)^j}=0
\end{equation*}
which proves that $a_0<\beta<a_0+1$. Also, when $2\leq n\leq 3$ we see that \eqref{Parrycon} is always satisfied.
\end{proof}

\subsection{Three important examples of finite Parry sequences}

\begin{enumerate}
\item The first example we consider is
\begin{equation}\label{dec4} a_0\geq a_1\geq \cdots \geq  a_{n-1}\geq 1.
\end{equation} 
Brauer \cite{Br} proved that \eqref{dec4} is also sufficient for $p(z)$ to be Parry-Pisot, which will play an important role in Theorem \ref{thm:pisot}.  For completeness, we give a short alternative proof of Brauer's result in Lemma \ref{brauer}.

\item Another example where \eqref{eq:a0a1akassump} and \eqref{Parrycon}  are satisfied is when there exists $0\leq j\leq n-2$ such that
\begin{equation}\label{dec5}
a_0 \ge a_1 \ge \cdots \ge a_j > a_{j+k}; \quad 1 \le  j+1 \le j+k \le n-1; \quad a_0a_{n-1} \ge 1. 
\end{equation}

Here $p(z)$ might not be a Parry-Pisot polynomial anymore. On the other hand, if we replace \eqref{dec5} with the more restrictive condition $a_0\geq 2+ a_1+\cdots +a_{n-1}$, which implies \eqref{dec5} with $j=0$, then $p(z)$ is Parry-Pisot again, see Lemma \ref{brauer'}. 

\item A third example where \eqref{eq:a0a1akassump} and \eqref{Parrycon} hold is 
\begin{equation}\label{dec6}
 n\geq 4,\quad   a_0\geq a_{n-1}\geq 1,\quad  a_1=\cdots =a_{n-2}=0.
\end{equation}
 We also prove in Appendix \ref{appE2} that all the roots of the polynomial $p(z)=z^n-z^{n-1}-1$ converge to the unit circle when $n\to\infty$. One can also check that when $n=6j+5$ with $j\geq 0$, two roots of this polynomial are $e^{\pm i \pi/3}$, see also \eqref{reduc} when $j=0$.  
\end{enumerate}

\section{The Main Results}
 
\subsection{Asymptotics of $\mathscr{P}^r(F)$ for smooth $F$'s}
Let us introduce the following linearly independent  (due to Lemma \ref{lemmaParry}) characteristic functions
\begin{align*}
    \psi_1(x) \coloneqq \chi_{[0,1]}(x), \quad  
    \psi_2(x) \coloneqq \chi_{[0,b_1]}(x), \quad \ldots,\quad 
      \psi_{n}(x) \coloneqq \chi_{[0,b_{n-1}]}(x),
\end{align*}
 and
\begin{align}\label{dec7}
    u_1(x) = C_0 \bigg( \psi_1(x) + \frac{1}{\beta} \psi_2(x) + \cdots + \frac{1}{\beta^{n-1}}\psi_n(x)\bigg)> 0,
\end{align}
where $C_0>0$ is chosen such that $\int_{[0,1]} u_1(x) \d x =1$. It turns out \cite{Pa} (see also \eqref{cdh1}) that $u_1$ is the unique positive solution to $\mathscr{P}f=f$ in $L^1([0,1])$, whose integral equals one. 

We now formulate the first result of our paper: 
 
\begin{theorem}\label{thm:pisot}
    Let $n\geq 2$ and let $\{a_j\}_{j=0}^{n-1}$ be a finite Parry sequence whose corresponding $p(z)$ is Parry-Pisot as in Definition \ref{def1}. Let  $\mathscr{P}$ be the operator given in \eqref{eq:mathcalP}. There exists a function $v_1\neq 0$ with the properties: 
    \begin{itemize}
        \item $v_1$ is piecewise equal to polynomials of degree at most one,
        \item $\mathscr{P}v_1 =\beta^{-1} v_1$ and $\int_{[0,1]}v_1(x)\d x=0$,
    \end{itemize}
    and there exist $N_0\geq 1$, $\varepsilon>0$, such that for every $N\geq N_0$ there exists a $C<\infty$ such that for all $F\in C^N([0,1])$ and $r\geq 1$ we have:
    \begin{align*}
        \Big\Vert (\mathscr{P}^rF)(\cdot) - u_1(\cdot) \int_0^1 F \d x -\tfrac{\beta^{-r}}{4}\big (F(1)-F(0) \big )v_1(\cdot)\Big \Vert_{L^\infty([0,1])} \leq C \Vert F \Vert_{C^N([0,1])}\beta^{-(1+\varepsilon)r}.
    \end{align*}
\end{theorem}

The proof of this theorem is in Section \ref{sec3.3}. 
\begin{remark} 
    Theorem \ref{thm:pisot} generalizes the results of \cite{CS} where only the case $n=2$ was considered. It turns out (see Lemma \ref{pisot}) that $p(z)$ is always a Parry-Pisot   polynomial for all possible finite Parry sequences when $n\leq 3$.  If $n\geq 4$ this is no longer true, as we will see later on. 
    
    Nevertheless, other conditions like in \eqref{dec4} are sufficient for both having a  Parry sequence and for $p(z)$ to be a Parry-Pisot polynomial (see Lemma \ref{brauer}). The condition  $a_0\geq 2+a_1+\cdots +a_{n-1}$ also satisfies both requirements (see \eqref{dec5} and Lemma \ref{brauer'}).
\end{remark}

\begin{remark}
 The estimate we obtain in Theorem \ref{thm:pisot} does not hold when $F$ is not smooth enough, for example when we allow functions of bounded variations, 
 a case in which the character of the problem changes completely. The best we can do in the Parry-Pisot case is an error bound going precisely like $\beta^{-r}$, see Corollary  \ref{coroll1}(iv).
\end{remark}
\begin{remark}
    The case in which the coefficients $\{a_j\}_{j\geq 0}$ become eventually periodic but not zero, remains open.  
\end{remark}

 When $p(z)$ is not Parry-Pisot, we have to consider $\W$ as an operator on functions with bounded variation and in general we cannot hope for error bounds better than $\beta^{-\alpha r}$ with $0<\alpha<1$. But when $F$ is a polynomial and $p(z)$ is ``barely" not Parry-Pisot because it has roots on the unit circle, we can do much better, as we can see in the following example. 
\begin{theorem}\label{nopisot2}
   Let $n=4$, $a_0=3$, $a_1=2$, $a_2=0$ and $a_3=2$, with $3<\beta<4$.  Then: 
   \begin{enumerate}
   \item 
   There exist two functions $u_2$ and $v_1$ with the properties: 
    \begin{itemize}
        \item $u_2$ is piecewise constant and $v_1$ is piecewise equal to a polynomial of degree at most one,
        \item $\mathscr{P}u_2 =-\frac{1}{\beta} u_2$, $\mathscr{P}v_1 =\frac{1}{\beta} v_1$ and $\int_{[0,1]}u_2(x)\d x=\int_{[0,1]}v_1(x)\d x=0$.
    \end{itemize}
    \item If $q(x)$ is any polynomial of degree $M\geq 1$ on $[0,1]$ then there exist $\varepsilon>0$ and $C<\infty$ and two coefficients $C_{2,q}$ and $C_{3,q}$ depending on $q$ such that for all $r\geq 1$ we have
   {\small \begin{align*}
        &\Big\Vert (\mathscr{P}^r q)(\cdot) - u_1(\cdot) \int_0^1 q(x) \d x  - \tfrac{(-1)^r}{\beta^r} C_{2,q} u_2(\cdot) -\tfrac{1}{\beta^{r}} C_{3,q} v_1(\cdot) \Big \Vert_{L^\infty} \leq C \Vert q \Vert_{C^M}\,\beta^{-(1+\varepsilon)r}.
    \end{align*} }
    \end{enumerate}
\end{theorem}

This theorem is proved in Section \ref{sec3.5}. 

\subsection{The spectrum of the transfer operator when restricted to $L^p([0,1])$ spaces}
Let $\{a_j\}_{j=0}^{n-1}$ be any finite Parry sequence with 
\begin{equation*}
    b_1=T_\beta(1)=\beta-a_0=\frac{a_1}{\beta}+\cdots+\frac{a_{n-1}}{\beta^{n-1}}\in (0,1).
\end{equation*}

\begin{theorem}\label{thm:eigvaleigfunc}
The following statements hold true: 
\begin{enumerate}
 \item { For any $g\in C_0^\infty((0,1))$ we define $\psi_{0}(x)=g(\beta x)\chi_{[0,b_1/\beta]}(x)-g(\beta x-a_0)\chi_{[a_0/\beta,1]}(x)$. Then  
\begin{align}\label{eq:betaPpsi0equalzero}
    (\mathscr{P}\psi_{0})(x) = 0 \quad \text{for a.e.}\quad  x \in [0,1].
\end{align} }
\item Let $u_1$ be as in \eqref{dec7}. Then $\tilde{K}_p \coloneqq {u_1}^{1/p} K \frac{1}{{u_1}^{1/p}}$ is a non-surjective isometry on $L^p([0,1])$ for all $1\leq p\leq \infty$. Moreover, the spectrum of $\tilde{K}_p$ and $K$ equals the closed unit disk. 
\item 
Let $1\leq p,p'\leq\infty$ and $\frac{1}{p}+\frac{1}{p'}=1$. For every $\vert z \vert < 1$ let us define  
\begin{align*}
    \psi_z \coloneqq u_1^{1/p'} \bigg( \mathrm{Id} - z \tilde{K}_p \bigg)^{-1} \bigg (\frac{\psi_{0}}{u_1^{1/p'}}\bigg )
\in L^p([0,1]).
\end{align*}
Then $\mathscr{P}\psi_z=z\psi_z$.
\end{enumerate}
\end{theorem}
The proof of this theorem is given in Section \ref{sec3.6}. 
\subsection{The resolvent of the transfer operator when restricted to functions of bounded variation}

The next result is about  the resolvent and the spectrum of $\W$ when restricted to functions of (essential) bounded variation $BV([0,1])\subset L^\infty([0,1])$ where the norm is given by:   
\begin{equation}\label{beve}
\begin{aligned}
&\Vert f\Vert_{BV}:=\inf\big  \{{\rm Var}_0^1(g)+\Vert g\Vert_{L^\infty}\, :\, f=g\; \text{Lebesgue almost everywhere on}\;  [0,1]\big \}, \\
&\text{where}\quad  {\rm Var}_0^1(g):=\sup_{m\geq 1}\, \sup_{\{0=x_1<x_2<\cdots<x_m=1\}}\sum_{j=1}^{m-1}|g(x_{j+1})-g(x_j)|.
\end{aligned}
\end{equation} 

One may also use $\Vert f\Vert_{L^1}$ instead of $\Vert f\Vert_{L^\infty}$ in the definition of the $BV$ norm. The two definitions lead to equivalent norms because on $[0,1]$ we have the inequalities
\begin{equation*}
    \Vert f\Vert_{L^1}\leq \Vert f\Vert_{L^\infty},\quad \Vert f\Vert_{L^\infty}\leq {\rm Var}_0^1(f) +\Vert f\Vert_{L^1}.
\end{equation*} 

Let $\Pi$ be the rank $n$ orthogonal projection corresponding to the subspace 
\begin{equation*}
    {\rm Span}\big \{\chi_{[0,b_0]}, \ldots, \chi_{[0,b_{n-1}]}\big \}\subset L^2([0,1])
\end{equation*} 
where the $b_j$'s are as in \eqref{dec1}.  The matrix $A=\{\langle \chi_{[0,b_k]},\chi_{[0,b_j]}\rangle_{L^2}\}_{k,j=0}^{n-1}$ is symmetric and invertible. Define  
$\tilde{u}_j\in {\rm Range}(\Pi)$ by $\tilde{u}_j=\sum_{k=0}^{n-1} \big (A^{-1}\big )_{jk}\chi_{[0,b_k]}$. Then $\int_0^1 \tilde{u}_j(t) \chi_{[0,b_m]}(t) \d t =\delta_{jm}$ and   
\begin{equation}\label{proj-pi}
    (\Pi f)(x)=\sum_{j=0}^{n-1} \chi_{[0,b_j]}(x) \int_0^1 \tilde{u}_j(t) f(t) \d t.
\end{equation}
\begin{theorem}\label{thm:mero}
Both $\W$ from \eqref{eq:mathcalP} and   $\Pi$ from \eqref{proj-pi} leave  $BV([0,1])$ invariant, are bounded on this space and $\W\Pi=\Pi \W\Pi$. Let $\mathcal{P}_n$ be the matrix from \eqref{Pn} corresponding to $\beta \Pi \W\Pi$. Then  on $BV$, the difference \begin{equation*}
    \big (\W-z)^{-1}-\Pi (\beta^{-1}\mathcal{P}_n -z\big )^{-1}\Pi 
\end{equation*} 
has an analytic extension from  $|z|>\Vert \W\Vert_{BV}\geq 1$ to $|z|>\beta^{-1}$. Thus $(\W-z)^{-1}$ is meromorphic on $|z|>\beta^{-1}$ with finite rank poles at all points $\zeta$ for which $\lambda=\zeta \beta$ (with $|\lambda|>1$) is a root of the Parry polynomial $p(z)$ from \eqref{dec2'}.  

Moreover, the essential spectrum of $\W$ is the whole disk $|z|\leq \beta^{-1}$, where every $z$ with $|z|<\beta^{-1}$ is an eigenvalue with  infinite multiplicity. 
\end{theorem}

The proof of this theorem is given in Section \ref{sec-mero}. 
\subsection{Asymptotics of $\W^r(f)$ when $f\in BV([0,1])$}

Here $\W$ is again considered as an operator on $BV([0,1])$. Let $\Pi$ be the projection from \eqref{proj-pi}, together with the matrix $\mathcal{P}_n$ associated to $\beta \Pi\W\Pi$ and its eigenvalues $\{\beta\equiv \lambda_1,\lambda_2,\ldots,\lambda_n\}$ from Lemma \ref{lemmahc1}, which are also the roots of the Parry polynomial $p(z)$ from \eqref{dec2'}. From Theorem \ref{thm:mero} we know that the essential spectrum of $\W$ restricted to $BV([0,1])$ equals the disk $|z|\leq \beta^{-1}$ and is pure point if $|z|<\beta^{-1}$. Also, from Theorem \ref{thm:mero} we know that all the numbers $\zeta_j:=\lambda_j \beta^{-1}$ with $|\zeta_j|>\beta^{-1}$ are discrete eigenvalues of $\W$. We have that $\zeta_1=1$ is always a simple eigenvalue, and when $p(z)$ is Parry-Pisot then this is the only discrete eigenvalue. 

\begin{corollary}\label{coroll1}
 Let $Q_j$ be the Riesz projection on the generalized eigenspace corresponding to the eigenvalue $\zeta_j$ with $|\zeta_j|>\beta^{-1}$ and let $Q = \sum_{|\zeta_j| > \beta^{-1}} Q_j$.  Let $1\leq m_j<n$ be the finite rank of $Q_j$. 
 \begin{enumerate}
 \item We have 
 \begin{equation*}
     \W^rQ = Q_1+\sum_{1>|\zeta_j| > \beta^{-1}} \zeta_j^r \big (1 + (N_j/\zeta_j)\big )^r Q_j
 \end{equation*}
 where $N_j:=(\W-\zeta_j)Q_j$ obeys $N_j^{m_j}=0$ and 
 \begin{equation*}
     \big (1+N_j/\zeta_j\big )^r = \sum_{l=0}^{m_j-1} (N_j/\zeta_j)^l \binom{r}{l}.
 \end{equation*} 

 \item Let $0<\alpha<1$. Then for all $r\geq 1$ we have
 \begin{equation}\label{finlacom}
 \W^r = Q_1+\sum_{1>|\zeta_j| > \beta^{-1}} \zeta_j^r \big (1 + (N_j/\zeta_j)^r \big ) Q_j + \mathcal{O}(\beta^{-\alpha r}), 
 \end{equation}
where $\mathcal{O}(\beta^{-\alpha r})$ is in the BV norm. We also have 
\begin{equation}\label{juli10}         
Q_1 f = u_1 \int_0^1 f(x)\d x.
\end{equation}
\item Assume that all eigenvalues obeying $1> |\zeta_j| > \beta^{-1}$ have multiplicity $m_j=1$. There exist $u_j\in {\rm Range}(\Pi)$ and  $\hat{U}^{(j)}\in BV([0,1])^*$ with $(Q_j f)(x)=u_j(x)\, \hat{U}^{(j)}(f)$  and $\hat{U}^{(j)}(u_k) = \delta_{jk}$ such that given $0<\alpha<1$ there exists $C_\alpha>0$ such that for all $f\in BV([0,1])$ and all $r\geq 1$ we have
\begin{equation}\label{juli23}
\Big \Vert (\W^r f)(\cdot)- u_1(\cdot) \int_0^1 f(t)\d t -\sum_{1>|\zeta_j| > \beta^{-1}} \zeta_j^r u_j(\cdot) \, \hat{U}^{(j)}(f)\Big \Vert_{BV}\leq C_\alpha \Vert f\Vert_{BV} \beta^{-\alpha r}.
\end{equation}
\item If $p(z)$ is Parry-Pisot, then there exists $C>0$ such that for all $f\in BV([0,1])$ and all $r\geq 1$ we have
\begin{equation*}
    \Big \Vert (\W^r f)(\cdot)- u_1(\cdot) \int_0^1 f(t)\d t \Big \Vert_{BV}\leq C \Vert f\Vert_{BV} \beta^{-r}.
\end{equation*}
\end{enumerate}
\end{corollary}

The proof of this corollary is given in Section \ref{proof-coro}. 

\subsection{The eigenvectors of $\W^*$ in $BV([0,1])^*$}
Here we want to give an ``explicit" formula for the functionals $\hat{U}^{(j)}$ from \eqref{juli23}. Before we can do that we need to introduce some notation, which will also play an important role in several other proofs. 

We split the interval $[0,1]$ as
\begin{align*}
    t_{k_1}^{j_1} := 
    \begin{cases}
\frac{j_1}{\beta} & 0 \leq j_1 \leq a_{0}-1  \qquad \textup{when } k_1 = 1 \\
\frac{a_0}{\beta} + \frac{j_1}{\beta^2} & 0 \leq j_1 \leq a_{1}-1 \qquad \textup{when } k_1 = 2 \\
\frac{a_0}{\beta} + \frac{a_1}{\beta^2} + \frac{j_1}{\beta^3} & 0 \leq j_1 \leq a_{2}-1 \qquad \textup{when } k_1 = 3 \\
\hspace{1cm}\vdots & \\
\frac{a_0}{\beta} + \frac{a_1}{\beta^2} + \cdots + \frac{a_{n-2}}{\beta^{n-1}} + \frac{j_1}{\beta^n} & 0 \leq j_1 \leq a_{n-1}-1 \qquad \textup{when } k_1 = n
    \end{cases}
\end{align*}
where we use the convention that if $a_{i-1}=0$, for $i \in  \{2,\ldots,n-1\}$, then we omit the case $k_1=i$, since this implies that we have no intervals of length $\beta^{-i}$. We also have 
\begin{align*}&t_k^{a_{k-1}-1} +\beta^{-k}=t_{k+1}^0,\quad a_{k-1}a_k>0,\quad 1\leq k\leq n-1,\\ 
& t_n^{a_{n-1}-1}+\beta^{-n}=1.
\end{align*} 
\begin{figure}[!h]
\begin{center}

\begin{tikzpicture}
  \draw[scale=10,-] (0, 0) -- (1, 0) node[right] {};
  \draw[-] (0, -0.1) -- (0, 0.1) node[below] {};
  \node[] at (0, -0.4)   (a) {$t_1^0=0$};
  \draw[-] (10, -0.1) -- (10, 0.1) node[below] {};
  \node[] at (10, -0.4)   (a) {$1$};

  \draw[-] (1, -0.1) -- (1, 0.1) node[below] {};
  \node[] at (1, -0.4)   (a) {$t_1^{1}$};
  
  \draw[-] (2, -0.1) -- (2, 0.1) node[below] {};
  \node[] at (2, -0.4)   (a) {$\cdots$};
  \draw[-] (3, -0.1) -- (3, 0.1) node[below] {};
  \node[] at (3, -0.4)   (a) {$t_2^0=\frac{a_0}{\beta}$};

  \draw[-] (3.7, -0.1) -- (3.7, 0.1) node[below] {};

   \draw[-] (4.2, -0.1) -- (4.2, 0.1) node[below] {};
  \node[] at (4.2, -0.4)   (a) {$\cdots$};

    \draw[-] (5.1, -0.1) -- (5.1, 0.1) node[below] {};
  \node[] at (5.6, -0.4)   (a) {$t_3^0=\frac{a_0}{\beta}+\frac{a_1}{\beta^2}$};

  \draw[-] (5.7, -0.1) -- (5.7, 0.1) node[below] {};

    \draw[-] (6.3, -0.1) -- (6.3, 0.1) node[below] {};

  \node[] at (7.3, -0.4)   (a) {$\cdots$};

   \draw[-] (8.5, -0.1) -- (8.5, 0.1) node[below] {};
  \node[] at (8.5, -0.4)   (a) {$t_n^0$};

\draw[-] (9, -0.1) -- (9, 0.1) node[below] {};
  \node[] at (9, -0.4)   (a) {};

  \draw[-] (9.5, -0.1) -- (9.5, 0.1) node[below] {};
  \node[] at (9.5, -0.4)   (a) {$\cdots$};

  \draw[black, decorate, decoration={brace, amplitude=5, raise=5}] (0, 0) -- (1, 0) node[pos=0.5, above=8] {$\frac{1}{\beta}$};

  \draw[black, decorate, decoration={brace, amplitude=5, raise=5}] (3.0, 0) -- (3.7, 0) node[pos=0.5, above=8] {$\frac{1}{\beta^2}$};

  \draw[black, decorate, decoration={brace, amplitude=5, raise=5}] (5.1, 0) -- (5.7, 0) node[pos=0.5, above=8] {$\frac{1}{\beta^3}$};
  \draw[black, decorate, decoration={brace, amplitude=5, raise=5}] (8.5, 0) -- (9, 0) node[pos=0.5, above=8] {$\frac{1}{\beta^n}$};
  
\end{tikzpicture}
\caption{Plot of the first splitting of the interval $[0,1]$.}
\label{fig:splitting}
\end{center}
\end{figure}
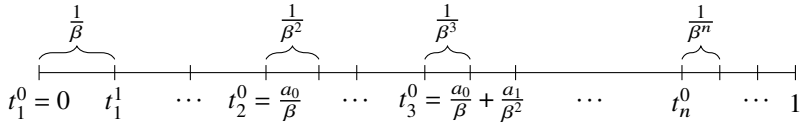

A second splitting of each subinterval is given by
\begin{align*}
    t_{k_1k_2}^{j_1j_2} := t_{k_1}^{j_1} + \beta^{-k_1} t_{k_2}^{j_2}
\end{align*}
and after $m$ splittings we have
\begin{align*}
    t_{k_1k_2\cdots k_m}^{j_1j_2 \cdots j_m} := t_{k_1k_2\cdots k_{m-1}}^{j_1j_2 \cdots j_{m-1}} + \beta^{-(k_1 + k_2 + \cdots + k_{m-1})} t_{k_m}^{j_m}.
\end{align*}
To simplify notation, we write
\begin{equation}\label{thekeys}
\begin{aligned}
    \bk_m &\coloneqq (k_1,k_2,\ldots,k_m), \\
    \bj_m &\coloneqq (j_1,j_2,\ldots,j_m), \\
    t_{\bk_m}^{\bj_m} &\coloneqq t_{k_1k_2\cdots k_m}^{j_1j_2 \cdots j_m}, \\
    \vert \bk_m \vert & \coloneqq k_1 + k_2+ \cdots + k_m,\\
    m&\leq \vert \bk_m \vert \leq mn.
\end{aligned}
\end{equation}

We will use the notation 
  \begin{equation*}
      \sum_{t_{\bk_m}^{\bj_m}}:=\sum_{k_1=1}^n \sum_{j_1=0}^{a_{k_1-1}-1} \cdots \sum_{k_m=1}^n \sum_{j_m=0}^{a_{k_m-1}-1}
  \end{equation*}
  with the same convention that if $a_{k-1}=0$ then $k$ does not appear in the sum.

\begin{theorem}\label{thm:left} Let  $\zeta=\lambda/\beta$ be an eigenvalue with  $1< |\lambda|\leq\beta$ and $\lambda$ is a root of the Parry polynomial which satisfies:  
\begin{equation}\label{aug1}
1=\sum_{j=1}^{n} a_{j-1}/\lambda^j=\sum_{k=1}^{n}\lambda^{-k}\sum_{j_k=0}^{a_{k-1}-1} 1.
\end{equation}
Let $f\in BV([0,1])$. For $m\geq 1$ we define
\begin{equation*}
    \hat{U}_m(f):=\sum_{t_{{\bk_m}}^{{\bf j}_m}} \lambda^{-|\bk_m|} f\big (t_{\bk_m}^{{\bf j}_m}\big ).
\end{equation*}
Then:
\begin{enumerate}[label=(\roman*)]
\item $\hat{U}(f):=\lim_{m\to\infty} \hat{U}_m(f)$ exists, it defines a bounded functional on ${BV}([0,1])$ and  
\begin{equation}\label{aug33}
    \hat{U}(\W f)=\zeta \, \hat{U}(f).
\end{equation}
When $\zeta=1$ then $\hat{U}(f)=\int_0^1 f(x) \d x$. 
\item Take any $y\in (0,1)$ with its greedy $\beta$-decomposition $y=\sum_{j\geq 1}\frac{y_{j-1}}{\beta^j}$ where $y_{j-1} = \lfloor\beta T^{j-1}(y)\rfloor$, $j\geq 1$. Then 
\begin{equation}\label{marvel}
\hat{U}(\chi_{[0,y]})=\sum_{j\geq 1}\frac{y_{j-1}}{\lambda^j}.
\end{equation}
Formula \eqref{marvel} coincides with the one obtained in \cite[Proposition 4.2]{Su}. If $h$ is an absolutely continuous function with support in $[0,1]$ then 
\begin{equation}\label{aug31}\hat{U}(h) = h(1) - \sum_{j=1}^\infty \lambda^{-j}\int_0^1y_{j-1}h'(y)dy.
\end{equation}
Also, by fixing $u= \sum_{j=0}^{n-1} \lambda^{-j-1} \chi_{[0,b_j]}$ so that $\W u=\zeta u$ (see \eqref{sep10}), then we have
\begin{equation}\label{sep11}
\hat{U}(u) =h_n'(\lambda),\quad h_n(z):=1-\sum_{k=0}^{n-1}a_k/z^{k+1}= z^{-n}p(z).
\end{equation}
Thus $\hat{U}(u)\neq 0$ if and only if $\lambda$ is a simple root of the Parry polynomial. 

\item If $\beta^{-1}<|\zeta|<1$ then 
\begin{equation*}
    (0,1)\ni y\mapsto \hat{U}(\chi_{[0,y]})\in \C
\end{equation*}
is nowhere differentiable. See also \cite[Theorem 4.3(v)]{Su}. 
\item If $\beta^{-1}<|\zeta|<1$ there exists no complex Radon measure $\mu$ so that $\hat{U}(f)=\int_0^1 f(x) \d \mu(x)$. See also \cite[Theorem 4.5]{Su}
\item Let $\beta^{-1}<|\zeta|<1$ be a simple eigenvalue with an eigenvector as in \eqref{sep10} where $c_1=1/\lambda$, and $\hat{U}(u_\zeta)=h_n'(\lambda)$ is given in \eqref{sep11}. There exists a sequence $\{C_s\}_{s\geq 1}\subset \C$ where each $C_s$ may be computed using a certain rank one Riesz projection given by a matrix of size $(s+1)n \times (s+1)n$ such that if $q$ is a polynomial with degree $M\geq 0$ then with $h_n$ from \eqref{sep11} we have
\begin{equation}\label{cdh291}
\begin{aligned}
   & \hat{U}(q) = \sum_{k_1=1}^n \sum_{j_1=0}^{a_{k_1-1}-1} \zeta^{-k_1}\bigg\{ \int_{t_{k_1}^{j_1}}^{t_{k_1}^{j_1}+\beta^{-k_1}} q(u) \d u  \\
        &+ h_n'(\lambda)\sum_{s=1}^{M+1} C_s \beta^{-k_1 s}\Big( q^{(s-1)}\big(t_{k_1}^{j_1}+\beta^{-k_1}\big)-q^{(s-1)}\big(t_{k_1}^{j_1}\big)\Big)  \frac{1}{s!}\Vert B_s \Vert_{L^1([0,1])} \bigg \}.
\end{aligned}
\end{equation}
\end{enumerate}
\end{theorem}  
The proof of this theorem is given in Section \ref{proof-left}.
\begin{remark}
    Using \eqref{aug31} in order to compute $\hat{U}(q)$ in \eqref{cdh291} would need to compute infinitely many $y_j$'s as functions of $y$ and integrate them. On the other hand, \eqref{cdh291} only requires a finite number of computations. 
\end{remark}

\section{Acknowledgements}
This work was supported by The Danish Council for Independent Research -- Natural Sciences,
grant DFF--10.46540/2032-00005B.

\section{Proofs of the Main Results}

\subsection{Proof of Theorem~\ref{thm:pisot}}\label{sec3.3}

From \eqref{thekeys} we see that  $t_{\bk_m}^{\bj_m}$ and $t_{\bk_m}^{\bj_m} + \beta^{- \vert \bk_m \vert}$ are ``consecutive points''. 
From Appendix \ref{sec:eulerbernoulliformula} we have that the rescaled Bernoulli approximation formula applied to each interval in the splitting, i.e. with $b=t_{\bk_m}^{\bj_m} + \beta^{- \vert \bk_m \vert}$ and $a=t_{\bk_m}^{\bj_m}$ is given by
\begin{equation}\label{eq:rescaledeulerbernoulli}
    \begin{aligned}
     F(y)
     &=   \beta^{|\bk_m|} \int_{t_{\bk_m}^{\bj_m} }^{t_{\bk_m}^{\bj_m} + \beta^{- \vert \bk_m \vert} } F(u) \d u  \\
     &\qquad +\sum_{s=1}^N \beta^{-(s-1)|\bk_m|}\big ( F^{(s-1)}(t_{\bk_m}^{\bj_m} + \beta^{- \vert \bk_m \vert} )-F^{(s-1)}(t_{\bk_m}^{\bj_m} )\big ) \frac{\widetilde{B}_s\Big( \beta^{|\bk_m|}\big (y-t_{\bk_m}^{\bj_m}\big )\Big)}{s!}\\
     &\qquad +\mathcal{O}\Big( \beta^{-|\bk_m|N}\Vert F^{(N)}\Vert_{C([0,1])}\Big),\quad \forall y\in \big [t_{\bk_m}^{\bj_m}\, ,\, t_{\bk_m}^{\bj_m} + \beta^{- \vert \bk_m \vert}\big ].
    \end{aligned}
    \end{equation} 
\begin{definition}\label{Giz}
Let $B_s$ be the Bernoulli polynomials for $s \in \mathbb{N}_0$. We define
{\footnotesize    \begin{align*}
        G_{s,\bk_m}^{\bj_m} (x) \coloneqq \frac{\beta^{|\bk_m|}}{\Vert B_s \Vert_{L^1([0,1])}}B_s \Big( \beta^{|\bk_m|} \Big(x-t_{\bk_m}^{\bj_m}\Big)\Big) \, \chi_{\big[t_{\bk_m}^{\bj_m},t_{\bk_m}^{\bj_m} + \beta^{- \vert \bk_m \vert}\big]}(x).
    \end{align*}}
    Moreover, for $j \in \{1,\ldots,n\}$ and with the $b_j$'s from \eqref{dec1} we define 
\begin{align*}
    H_{sn+j}(x):= \frac{1}{b_{j-1}}\frac{B_{s}\big(\frac{1}{b_{j-1}} x \big)}{\Vert B_{s} \Vert_{L^1([0,1])}} \chi_{[0,b_{j-1}]}(x),\quad H_1=\psi_1=\chi_{[0,1]}.
\end{align*}
\end{definition}
  We now let $y$ in \eqref{eq:rescaledeulerbernoulli} to belong to $[0,1]$, the price being that we need to sum up all possible contributions coming from consecutive intervals.  Recall the notation 
  \begin{equation*}
      \sum_{t_{\bk_m}^{\bj_m}}:=\sum_{k_1=1}^n \sum_{j_1=0}^{a_{k_1-1}-1} \cdots \sum_{k_m=1}^n \sum_{j_m=0}^{a_{k_m-1}-1}
  \end{equation*}
  with the same convention that if $a_{k-1}=0$ then $k$ does not appear in the sum.
  
  Then given $m\geq 1$ we have $\vert \bk_m\vert \geq m$ and 
\begin{equation}\label{hor13}
    \begin{aligned}
        F(y) &= \sum_{t_{\bk_m}^{\bj_m}} \bigg\{  G_{0,\bk_m}^{\bj_m} (y)\, \int_{t_{\bk_m}^{\bj_m}}^{t_{\bk_m}^{\bj_m} + \beta^{- \vert \bk_m \vert}} F(u) \d u   \\
        &\qquad + \sum_{s=1}^N \beta^{-|\bk_m|s}\Big( F^{(s-1)}(t_{\bk_m}^{\bj_m} + \beta^{- \vert \bk_m \vert})-F^{(s-1)}(t_{\bk_m}^{\bj_{m}})\Big)  \frac{1}{s!}\Vert B_s \Vert_{L^1([0,1])} \, G_{s,\bk_m }^{\bj_m}(y)\bigg\} \\
        & \qquad +\mathcal{O}\Big(\beta^{-mN} \big \Vert F^{(N)}\Vert_{C([0,1])} \Big),\quad \forall y\in [0,1].
    \end{aligned}
    \end{equation}
By applying the transfer operator a sufficient number of times we can bring the Bernoulli polynomials from the small intervals in the splitting back to the full interval $[0,1]$, as shown in the following result (a proof of which can be found in Appendix \ref{sec:splittingofinterval}):
\begin{lemma}\label{lemmacrux}
Let $s \in \mathbb{N}_0$. Then for almost every $x\in [0,1]$ and 
  {\bf independently of the choice of} $\bj_m$ we have 
  \begin{align*}
        \Big(\mathscr{P}^{|\bk_m|} G_{s,\bk_m}^{\bj_m}\Big)(x) = {H}_{sn+1}(x).
    \end{align*}
\end{lemma}

{ The next step is to apply the above identity to \eqref{hor13} while controlling  the $L^\infty$ norm of the last error term. We use that $(a_0+1)^r/\beta^r$ is an upper bound for the norm of $\mathscr{P}^r$ on $L^\infty$ (see Lemma~\ref{lem:ks11} (iii)).  Since $1\leq a_0<\beta<a_0+1$ implies $(a_0+1)/\beta<2$, then we may write:
\begin{align}\label{cdh20}
   & \big(\mathscr{P}^r F\big)(y) = \sum_{t_{\bk_m}^{\bj_m}} \bigg\{  \big(\mathscr{P}^{r-|\bk_m|}H_1\big)(y) \int_{t_{\bk_m}^{\bj_m}}^{t_{\bk_m}^{\bj_m} + \beta^{- \vert \bk_m \vert}} F(u) \d u  \\\nonumber 
        &+ \sum_{s=1}^N \beta^{-|\bk_m|s}\Big( F^{(s-1)}\big(t_{\bk_m}^{\bj_m} + \beta^{- \vert \bk_m \vert}\big)-F^{(s-1)}\big(t_{\bk_m}^{\bj_{m}}\big)\Big)  \frac{1}{s!}\Vert B_s \Vert_{L^1([0,1])} \big(\mathscr{P}^{r-|\bk_m|} H_{sn+1}\big)(y)\bigg\} \\ \nonumber 
        & +\mathcal{O}\Big(2^r\beta^{-mN} \big \Vert F^{(N)}\Vert_{C([0,1])} \Big).
\end{align}
The first term is treated with Lemma~\ref{lem:ks11} $(iv)$, and also using that $\vert \bk_m\vert \leq nm$, which yields
\begin{align*}
    \sum_{t_{\bk_m}^{\bj_m}}  \big(\mathscr{P}^{r-|\bk_m|}H_1\big)(x) \int_{t_{\bk_m}^{\bj_m}}^{t_{\bk_m}^{\bj_m} + \beta^{- \vert \bk_m \vert}} \hspace{-0.5cm}F(u) \d u = u_1(x)\int_0^1 F(u)\d u + \mathcal{O}\bigg( \bigg(\frac{\vert \lambda_*\vert}{\beta}\bigg)^{r-mn}\Vert F \Vert_{C([0,1])}\bigg).
\end{align*}
We split the second term from \eqref{cdh20} in $s=1$ and $s \geq 2$. When $s = 1$ we use Lemma \ref{lemma26} along with the Mean Value Theorem to get
\begin{align*}
\sum_{t_{\bk_m}^{\bj_m}} & \beta^{-|\bk_m|}\Big( F\big(t_{\bk_m}^{\bj_m} + \beta^{- \vert \bk_m \vert}\big)-F\big(t_{\bk_m}^{\bj_{m}}\big)\Big)  \Vert B_1\Vert_{L^1([0,1])} \big(\mathscr{P}^{r-|\bk_m|} H_{n+1}\big)(x) \\
    &= 4^{-1}\beta^{-r} \big (F(1)-F(0) \big ) v_1(x)+\beta^{-m} \Vert F'\Vert_{C([0,1])} \mathcal{O}\bigg( \bigg(\frac{|\lambda_*|}{\beta}\bigg)^{r-nm}\bigg) .
\end{align*}

For $s \geq 2$ we again use the Mean Value Theorem, Lemma~\ref{lem:ks11} $(v)$, and $|\bk_m|\geq m$ to get
\begin{align*}
    \sum_{t_{\bk_m}^{\bj_m}} &\sum_{s=2}^N \beta^{-|\bk_m|s}\Big( F^{(s-1)}\big(t_{\bk_m}^{\bj_m} + \beta^{- \vert \bk_m \vert}\big)-F^{(s-1)}\big(t_{\bk_m}^{\bj_{m}}\big)\Big)  \frac{1}{s!}\Vert B_s \Vert_{L^1([0,1])} \big(\mathscr{P}^{r-|\bk_m|} H_{sn+1}\big)(y) \\
    &=\max_{\bk_m}\sum_{s=2}^N \beta^{-s\vert \bk_m\vert } \Vert F^{(s)}\Vert_{C([0,1])} \mathcal{O}\bigg( \bigg(\frac{1}{\beta}\bigg)^{r-\vert \bk_m\vert }\bigg) =\mathcal{O}\bigg( \bigg(\frac{1}{\beta}\bigg)^{r+m}\Vert F \Vert_{C^N([0,1])}\bigg).
\end{align*}
Combining all of this leads to
\begin{align*}
    \big(\mathscr{P}^r F\big)(x) &= u_1(x)\int_0^1 F(u)\d u + \frac{F(1)-F(0)}{4 \beta^r} v_1(x)  +  \mathcal{O}\bigg( \bigg(\frac{\vert \lambda_*\vert}{\beta}\bigg)^{r-mn}\Vert F \Vert_{C^N([0,1])}\bigg) \\ &\qquad + \beta^{-m} \Vert F'\Vert_{C([0,1])} \mathcal{O}\bigg( \bigg(\frac{|\lambda_*|}{\beta}\bigg)^{r-mn}\bigg) + \mathcal{O}\bigg( \bigg(\frac{1}{\beta}\bigg)^{r+m}\Vert F \Vert_{C^N([0,1])}\bigg) \\
    &\qquad +\mathcal{O}\Big(2^r\beta^{-mN} \big \Vert F^{(N)}\Vert_{C([0,1])} \Big).
\end{align*}
We will now show that there exists an $\epsilon>0$ as stated in the theorem, provided $N$ is large enough.  Let $\delta \in (0,1)$. We first consider the last error term from above, for which we ask that  $2^r \beta^{-mN} = \beta^{-r(1+\delta)}$. This gives
\begin{align*}
    m_{\delta,r,N}= r\frac{\ln(2) + (1+\delta)\ln(\beta)}{N\ln(\beta)}
\end{align*}
which is less than $r/n$ if $N$ is large enough and $r>n$. We may ignore the fact that $m_{\delta,r,N}$ might not be a natural number. 
Next, we consider the first error term for which we ask 
\begin{equation*}
    \vert \lambda_* \vert^{r-nm_{\delta,r,N}} \beta^{-r+nm_{\delta,r,N}} \leq \beta^{-r(1+\delta)}.
\end{equation*} 
This holds for all $r$ if  
\begin{align*} 
    \delta \leq \frac{\ln(\vert \lambda_*\vert^{-1}) \Big( 1 - \frac{n\ln(2\beta)}{N \ln(\beta)}\Big) - \frac{n\ln(2\beta)}{N}}{\Big( \frac{n\ln(\vert \lambda_* \vert^{-1})}{N} + \frac{n}{N}\ln(\beta) + \ln(\beta) \Big)},
\end{align*}
a condition which can be satisfied if $N\geq N_0$ with $N_0$ large enough and $\delta <\tfrac{\ln(\vert \lambda_*\vert^{-1})}{\ln(\beta)}$. 

The second error term has an extra factor $\beta^{-m}$ and is automatically less than the first one. The third error term is already of the form 
$\beta^{-r-m_{\delta,r,N}}$ which allows us to choose 
\begin{align*}
 \varepsilon =\min\bigg \{\delta,   \frac{\ln(2) + (1+\delta)\ln(\beta)}{N_0\ln(\beta)}\bigg \} .
\end{align*}
and we are done. \qed
}

\subsection{Proof of Theorem~\ref{nopisot2}}\label{sec3.5}
\begin{enumerate}[label=(\roman*)]
    \item Here $n=4$, $a_0=3$, $a_1=2$, $a_2=0$ and $a_3=2$. Since $a_0$ is strictly larger than the other coefficients, they form a finite Parry-sequence. {
The corresponding Parry polynomial is $p(z)=z^4-(3 z^3+2z^2+2)$ with $3<\beta<4$. Also, $p(-1)=0$ and $p'(-1)\neq 0$ thus $-1$ is a simple root. In particular, $p(z)$ is not Parry-Pisot. 
The function $f(x)=p(x)/x^3=x-3-2/x-2/x^3$ has the same zeros as $p$, and $f'(x) >0$ on $(-\infty,0)$, so $f$ and $p$ have no other negative roots besides $-1$. Since $\beta$ is the only positive root, the other two roots of $p$ must have a nonzero imaginary part and $z_3=\overline{z_4}$. Moreover, because the product of all four roots equals $-2$ we must have

\begin{equation*}
    \vert z_3 \vert ^2= \vert z_4 \vert^2=2/\beta
\end{equation*}
hence $\vert z_3 \vert=\vert z_4\vert <1$. We note that $\beta$ is nevertheless a ``plain" Pisot number via the non-Parry polynomial 
\begin{equation*}
    \frac{p(z)}{z+1}=z^3 - 4 z^2 + 2 z - 2,
\end{equation*}
whose roots are $\beta$, $z_3$ and $z_4$. 

This choice of coefficients corresponds to a restriction matrix $\mathcal{P}_4$ (see \eqref{Pn}) which has $\lambda_1=\beta$, $\lambda_2=-1$, while $|\lambda_3|=|\lambda_4|<1$. 

 Here we have $\vert \lambda_*\vert=1$. The matrix ${\mathcal{A}}_{4,1}$ (the first block in \eqref{cdh30}) has eigenvalues $1$, $-1/\beta$, and $|\lambda_3/\beta|=|\lambda_4/\beta|<\beta^{-1}$. One can compute the $4\times 4$ Riesz projection corresponding to the eigenvalue $-1/\beta$, and let $p_{j1}$ with $1\leq j\leq 4$ be the elements of the first column of this projection. Then it turns out that 
$u_2(x)=\sum_{j=1}^4 p_{j1} \, H_j(x)$
is a piecewise constant eigenfunction corresponding to $-1/\beta$ for which we also have 
\begin{equation*}
    (\mathscr{P}^r H_1)(x) =u_1(x)+ (-1)^r \beta^{-r} u_2(x) +\mathcal{O}\bigg(\bigg(\frac{\vert z_3\vert}{\beta}\bigg)^r\bigg).
\end{equation*}

The function $v_1$ is constructed like in Theorem \ref{thm:pisot} with $\mathscr{P}$ restricted to the larger 
subspace generated by $\{H_j\}_{j=1}^8$ and it corresponds to the simple eigenvalue equal to $+1/\beta$ coming from the second block of the matrix in \eqref{cdh30}. 

\item If $q$ is a polynomial of degree $M$ then by choosing $n=4$, $m=1$ and $N=M+1$ in formula \eqref{cdh20} we see that the error term is identically zero and 
\begin{align}\label{cdh29}
   & \big(\mathscr{P}^r q\big)(y) = \sum_{k_1=1}^n \sum_{j_1=0}^{a_{k_1-1}-1}\bigg\{  \big(\mathscr{P}^{r-k_1}H_1\big)(y) \int_{t_{k_1}^{j_1}}^{t_{k_1}^{j_1}+\beta^{-k_1}} q(u) \d u  \\\nonumber 
        &+ \sum_{s=1}^{M+1} \beta^{-k_1 s}\Big( q^{(s-1)}\big(t_{k_1}^{j_1}+\beta^{-k_1}\big)-q^{(s-1)}\big(t_{k_1}^{j_1}\big)\Big)  \frac{1}{s!}\Vert B_s \Vert_{L^1([0,1])} \big(\mathscr{P}^{r-k_1} H_{4s+1}\big)(y)\bigg \}.
\end{align}
Note that the term with $s=M+1$ is zero because $q^{(M)}$ is constant. We have reduced the problem to the analysis of functions of the type $\mathscr{P}^{r-k_1} H_{4s+1}$ with $0\leq s\leq M$ and $k_1\in \{1,2,4\}$ ($k_1=3$ is missing because $a_2=0$). 

Thus  $\mathscr{P}^r q$ can be written as a linear combination of  eigenvectors of a  matrix as in \eqref{cdh30} where $\nu=M+1$ and $n=4$. The eigenvalue $1$ only comes from $\mathscr{P}^{r-k_1}H_1$, while $\mathscr{P}^{r-k_1} H_{4s+1}$ with $s\geq 1$ have no components in the subspace spanned by $u_1$ because all these functions integrate to zero. We can in principle identify and add together all the components living in the spectral subspace of the simple eigenvalue $-1/\beta$, and this would generate the constant $C_{2,q}$. Also, we can do the same for the eigenvalue $1/\beta$ and identify $C_{3,q}$, and so on. Finally, because all other eigenvalues are in absolute value less than $1/\beta$ we conclude that the function
\begin{equation*}
    (\mathscr{P}^r q)(\cdot) - u_1(\cdot)\int_0^1 q(u)\d u  -C_{2,q}(-\beta)^{-r} u_2(\cdot) -C_{3,q} \beta^{-r} v_1(\cdot)
\end{equation*}
will decay like $\beta^{-r(1+\varepsilon)}$. \qed  }
\end{enumerate}

\subsection{Proof of Theorem~\ref{thm:eigvaleigfunc}}\label{sec3.6}
(i). Let us first prove \eqref{eq:betaPpsi0equalzero}. Using that $\beta-a_0=b_1$ we have 
\begin{align*}
    \beta (\mathscr{P}\psi_0)(x)=g(x)\chi_{[0,b_1]}(x)-g(x)\chi_{[0,b_1]}(x)=0.
\end{align*} 

{

(ii). Now let us show that $\tilde{K}_p$ is an isometry; note that $K$, $\mathscr{P}$ and $\tilde{K}_p$ commute with the complex conjugation, thus it is enough to work with real functions. Then: 
\begin{align*}
    \int_0^1 \vert\tilde{K}_pf\vert^p(x) \d x &= \int_0^1 u_1(x) \bigg( \frac{\vert f\vert }{{u_1}^{1/p}}\bigg)^p\big( T_\beta(x)\big) \d x = \int_0^1 u_1(x)   K\bigg( \bigg( \frac{\vert f\vert }{{u_1}^{1/p}}\bigg)^p\bigg)(x) \d x \\
    &= \int_0^1 \vert f\vert ^p(x) \d x,
\end{align*}
where in the second equality we have used that $K$ is acting as a change of variables, while in the last equality we used the duality \eqref{dec12}  between $K$ and $\mathscr{P}$ and that $\mathscr{P} u_1=u_1$. 

Now let us prove that $\tilde{K}_p$ is not surjective, by showing that there is no $f\in L^p$ such that $\tilde{K}_pf=\psi_0$ with $\psi_0$ from (i). Assume the contrary; then: 
\begin{equation*}
    0<\int_0^1 u_1(x)^{-1/p}\psi_0^2(x)\d x=\int_0^1 \psi_0(x) \big (K(u_1^{-1/p}f)\big )(x)\d x=0,
\end{equation*}
where in the last equality we again used duality \eqref{dec12}  between $K$ and $\mathscr{P}$ and that $\mathscr{P} \psi_0=0$. We conclude that $\tilde{K}_p$ is a non-surjective isometry on a Banach space, thus  using \cite[Lemma 3.1]{CHM} we get that the spectrum of $\tilde{K}_p$ equals the closed unit disk. Finally, the same holds for $K$ because it is conjugated with $\tilde{K}_p$. 

(iii). If $p=2$ we have $K^*=\mathscr{P}$ and 
\begin{align*}
    \tilde{K}_2^* = \frac{1}{\sqrt{u_1}} \mathscr{P} \sqrt{u_1}, \qquad \tilde{K}_2^*\tilde{K}_2 =\mathrm{Id}.
\end{align*}
The last identity can be written as $u^{-1/2}\mathscr{P} u_1 K u_1^{-1/2} = \mathrm{Id}$ on $L^2$ or $\mathscr{P} u_1 K f=u_1 f$ for any smooth $f$. This last identity can be extended to any $f\in L^p$ with $1\leq p<\infty$. It also holds on $L^\infty([0,1])$ because this space is included in  $L^2([0,1])$. If $\frac{1}{p}+\frac{1}{p'}=1$ then this leads to
\begin{equation}\label{cdh35}
   u_1^{-1/p'}\mathscr{P} u_1^{1/p'} 
   u_1^{1/p} K u_1^{-1/p}=u_1^{-1/p'}\mathscr{P} u_1^{1/p'} \tilde{K}_p= \mathrm{Id} \quad \text{on}\quad  L^{p},\quad 1\leq p\leq \infty. 
\end{equation}

If $\vert z\vert <1$ the operator $z\tilde{K}_p$ has a norm less than $1$ and $(\mathrm{Id}-z\tilde{K}_p)^{-1}$ exists as a Neumann series in $L^p$. Then
\begin{align*}
    \psi_z =\sum_{j=0}^\infty z^j u_1^{1/p'} \tilde{K}_p^j \bigg(\frac{1}{u_1^{1/p'}}\psi_0\bigg) = \psi_0 + \sum_{j=1}^\infty z^j u_1^{1/p'}\tilde{K}_p^j \bigg(\frac{1}{u_1^{1/p'}}\psi_0\bigg).
\end{align*}
Moreover:
\begin{align*}
    \mathscr{P}\psi_z &= \mathscr{P}\psi_0 + z \sum_{j=1}^\infty z^{j-1} u_1^{1/p'} \Big ( u_1^{-1/p'}\mathscr{P} u_1^{1/p'} \Big )\bigg( \tilde{K}_p^{j} \bigg( \frac{1}{u_1^{1/p'}} \psi_0 \bigg) \bigg) \\ 
    &= 0 + z \sum_{j=1}^\infty z^{j-1} u_1^{1/p'} \tilde{K}_p^{j-1} \bigg( \frac{1}{u_1^{1/p'}} \psi_0 \bigg) = z \psi_z,
\end{align*}
where in the second equality we have used \eqref{cdh35}. This shows that every $z \in \mathbb{C}$ with $\vert z \vert  < 1$ is an eigenvalue of $\mathscr{P}$ in any $L^p$. 
}

 \qed

 \subsection{Proof of Theorem~\ref{thm:mero}}\label{sec-mero}

Let us first deal with the spectrum in $|z|<\beta^{-1}$ and show that any such $z$ is an eigenvalue in $BV([0,1])$. Let $\tilde{K}_1$ be the isometry from the previous theorem with $p=1$.  For every $j\geq 1$ and every $f\in BV([0,1])$ we have the identity 
\begin{equation*}
    (\tilde{K}_1^j f)(x)=\frac{u_1(x)}{u_1(T_\beta(x))} (\tilde{K}_1^{j-1} f)(T_\beta(x))
    =\cdots =\frac{u_1(x)}{u_1(T_\beta^j(x))} f(T_\beta^j(x)).
\end{equation*}
Let $\psi_0$ be the function obeying \eqref{eq:betaPpsi0equalzero}. Then: 
\begin{equation*}
    (\tilde{K}_1^j \psi_0)(x)=\frac{u_1(x)}{u_1(T_\beta^j(x))} \Big ( g\big (\beta T_\beta^j(x)\big )\chi_{[0,b_1/\beta]}(T_\beta^j(x))-g\big (\beta T_\beta^j(x)-a_0\big )\chi_{[a_0/\beta,1]}(T_\beta^j(x))\Big ).
\end{equation*}
We now show that the total variation of this function is of order $\beta^j$. Since all the involved functions are in $L^\infty$, it is enough to estimate the total variation of each factor. 

When $x\in [0,1]$, the range of $T_\beta$ between two discontinuities can either be $[0,1)$ or $[0,b_1]$. The range of $T_\beta^2$ can be $[0,1)$, $[0,b_1]$, or $[0,b_2]$. The range of $T_\beta^{n-1}$ can be one of the intervals $[0,1)$ and $[0,b_k]$ with $1\leq k\leq n-1$. No other ``vertical" variations of $T_\beta^j$ are possible, no matter what $j\geq 1$ is. Now let $c=\min\{b_1,\ldots,b_{n-1}\}\in (0,1)$. The map $T^j_\beta$ has a local slope equal to $\beta^j$, thus  the smallest possible interval between two consecutive discontinuities of $T_\beta^j$ has a length of $c/\beta^j$. Hence if $N_j$ is the number of intervals whose endpoints are  consecutive discontinuities of $T_\beta^j$ on $[0,1]$, we must have $N_j c/\beta^j\leq 1$, or $N_j\leq c^{-1}\beta^j$. 

Thus the total variation of each factor grows, up to a constant, as $\beta^j$. Since $|z|<\beta^{-1}$, the series 
$\psi_z =\sum_{j=0}^\infty z^j  \tilde{K}_1^j (\psi_0)$
converges in $BV([0,1])$ and $\W \psi_z=z \psi_z$. Also, we can construct countably many linearly independent such eigenfunctions given the freedom we have in choosing the function $g$.

We now investigate the domain $|z|>\beta^{-1}$ and show that it can only contain discrete spectrum. For that we will need three lemmas:

\begin{lemma}\label{lem1}
    Let $f\in BV([0,1])$ and let $g_1= f(1)\beta^{-1} \chi_{[0,b_1]}$ with $b_1$ as in \eqref{dec1}. Then 
    \begin{equation}\label{X1}
      {\rm Var}_0^1\big (\W f -g_1\big )\leq \beta^{-1} {\rm Var}_0^1(f).   
    \end{equation}
\end{lemma}

The proof of this lemma is in Appendix \ref{sec-lem1}.

\begin{lemma}\label{lemmap2} Let $|\lambda_*|<\beta$  denote the ``second largest" absolute value among the roots of  $p(z)$ from \eqref{dec2'}.
     There exists $C>0$ such that for all $|\lambda_*|<\beta$, all $r\geq 1$ and all $f\in BV([0,1])$ with $\int_0^1 f(x) \d x=0$  we have the following alternatives:
    \begin{equation}\label{jun1}
    \begin{aligned}
        \Vert \mathscr{P}^r f \Vert_{BV}\leq C\,  {\rm Var}_0^1(f)\,\cdot \left\{ \begin{matrix}
         \beta^{-r} \quad \text{if}\quad |\lambda_*|<1,
          \\  r^{n}\, \beta^{-r} \quad  \text{if}\quad  |\lambda_*|=1,\\  r^{n-1}\,
          (|\lambda_*|/\beta)^{r}\quad  \text{if}\quad  1<|\lambda_*|<\beta. 
        \end{matrix}
        \right . 
        \end{aligned}
    \end{equation}
\end{lemma}
This lemma is proved in Appendix \ref{sec-lemmap2}. The last technical result we need is the following:

\begin{lemma}\label{lemmap1}
     There exists $C>0$ such that for all $f\in BV([0,1])$ and for all $r\geq 1$ there exists $h_r\in {\rm Span}\big \{\chi_{[0,b_j]}:\, 0\leq j\leq n-1\big \}$ with the $b_j$'s from \eqref{dec1} such that 
     \begin{equation*}
         {\rm Var}_0^1\big (\mathscr{P}^rf-h_r\big ) \leq C\, \Vert f\Vert_{BV}\, \beta^{-r}.
     \end{equation*}
 \end{lemma}
 The proof of this lemma is given in Appendix \ref{sec-lemmap1}. Now we are ready to prove the theorem. Let us fix $\epsilon>0$ and choose $z$ with $|z|\geq \beta^{-1}+\epsilon$ which is not in the spectrum of $\beta^{-1} \mathcal{P}_n$ where $\mathcal{P}_n$ is as in \eqref{Pn}, or equivalently, $\beta z$ is not a root of the Parry polynomial; these roots are analyzed in Lemma \ref{lemmahc1}. We first want to show that there exists $r_0>1$, independent of $z$,  such that $z^{r_0}$ belongs to the resolvent set of $\W^{r_0}$, and after that we show that $z$ itself is in the resolvent set of $\W$. 

 We remind that $\Pi$ is the finite rank orthogonal projection in $L^2([0,1])$ corresponding to the subspace generated by the functions $\chi_{[0,b_j]}$ with $0\leq j\leq n-1$, see \eqref{proj-pi}. Because these functions are all characteristic functions (thus in $BV([0,1])$), $\Pi$ remains a projection even when considered on $BV([0,1])$.  
We know from Appendix \ref{sec:invsubspace} that $\W$ leaves the range of $\Pi$ invariant, hence we may write (with a slight abuse of language) 
\begin{equation}\label{hol}
\W^r-z^r=\begin{bmatrix} (\beta^{-1}\mathcal{P}_n)^r-z^r   &  \Pi \W^r (1-\Pi)\\
0 & (1-\Pi)\W^r (1-\Pi) -z^r(1-\Pi)
\end{bmatrix},\quad r\geq 1.
\end{equation}

On the range of $1-\Pi$ we may assume that the functions integrate to zero because the constant function $\chi_{[0,1]}$ is in the range of $\Pi$. Since $\W$ preserves the value of the integral (see Lemma \ref{lem:ks11}(i)), all the functions in the range of $\W^r (1-\Pi)$ will also integrate to zero. If $\int_0^1 f(x)\d x=0$ then ${\rm Var}_0^1(f)\leq \Vert f\Vert_{BV}\leq 2 {\rm Var}_0^1(f)$, hence Lemma \ref{lemmap1} implies that 
\begin{equation}\label{fejl3}
    \Vert (1-\Pi)\W^{r}(1-\Pi)\Vert_{BV}\leq C \beta^{-r},\quad \forall r\geq 1.
    \end{equation}
    Now there exists some $r_0>1$ such that  
    \begin{equation*}
        \vert z^{r_0} \vert =\beta^{-{r_0}}(1+|\beta z|-1)^{r_0}\geq \beta^{-r_0}(1+\beta \epsilon)^{r_0}> 2C\beta^{-r_0}
    \end{equation*}
    with the same $C$ as in \eqref{fejl3} and  
    \begin{equation*}
        A(z):=(1-\Pi)\W^{r_0} (1-\Pi) -z^{r_0}(1-\Pi)
    \end{equation*}
    is invertible with  
    \begin{equation}\label{sep8}
    A(z)^{-1}=-\sum_{k\geq 0}z^{-r_0(k+1)}\Big ((1-\Pi)\W^{r_0} (1-\Pi)\Big )^k
    \end{equation}
    in the Banach space $(1-\Pi)BV([0,1])$. This implies that 
\begin{equation}\label{hol'}
(\W^{r_0}-z^{r_0})^{-1}=\begin{bmatrix} \big ((\beta^{-1}\mathcal{P}_n)^{r_0}-z^{r_0} \big )^{-1} &  -\big ((\beta^{-1}\mathcal{P}_n)^{r_0}-z^{r_0}\big )^{-1}\Pi \W^{r_0} (1-\Pi)A(z)^{-1}\\
0 & A(z)^{-1}
\end{bmatrix}.
\end{equation}

The polynomial $\xi^{r_0}-z^{r_{0}}$ can be factorized as $(\xi-z) \sum_{k=0}^{r_0-1}\xi^k z^{r_0-1-k}$ which leads to 
\begin{equation}\label{hol4}
\begin{aligned}
(\W-z)^{-1}&=\Big (\sum_{k=0}^{r_0-1}z^{r_0-1-k}\W^k\Big ) \, (\W^{r_0}-z^{r_0})^{-1}\\
&=\begin{bmatrix} \big ((\beta^{-1}\mathcal{P}_n)-z \big )^{-1} &  -\big ((\beta^{-1}\mathcal{P}_n)-z\big )^{-1}\Pi \W^{r_0} (1-\Pi)A(z)^{-1}\\
0 & (1-\Pi)\sum_{k=0}^{r_0-1}z^{r_0-1-k}\W^k(1-\Pi)A(z)^{-1}
\end{bmatrix}
\end{aligned}
\end{equation}
which shows that $z$ is in the resolvent set of $\W$ when $\W$ is restricted to $BV([0,1])$. This proves that the only singularities of $(\W-z)^{-1}$ on $|z|> \beta^{-1}$ are of the type $\lambda/\beta$ where $\lambda$ is a root of the Parry polynomial which obeys $1<\vert \lambda\vert \leq \beta$.

\qed

\subsection{Proof of Corollary \ref{coroll1}}\label{proof-coro}
The key starting point is formula \eqref{hol4}. Choose $\epsilon>0$ small enough such that all eigenvalues of $\beta^{-1}\mathcal{P}_n$ obeying $1\geq |\zeta_j|>\beta^{-1}$ belong to the domain $|z|>\beta^{-1}+2\epsilon$. Then the operator $A(z)^{-1}$ from \eqref{sep8} is analytic on $|z|>\beta^{-1}+\epsilon$ as a convergent Neumann series, for some large enough $r_0$ depending on $\epsilon$. Denoting by $q_j$ the Riesz projections of $\beta^{-1}\mathcal{P}_n$ corresponding to the eigenvalues $1\geq |\zeta_j|>\beta^{-1}$, from \eqref{hol4} we obtain: 
\begin{equation}\label{chiuj}
Q_j=\Pi q_j \Pi -\Pi q_j\Pi \W^{r_0} A(\zeta_j)^{-1}(1-\Pi).
\end{equation}
The norm estimate in \eqref{finlacom} which needs an $0<\alpha<1$ is a consequence of the existence of the essential spectrum, and the fact that we use the Gelfand formula relating $\lim_{r\to\infty}\Vert (\W (1-Q))^r\Vert^{1/r}$ with the  spectral radius of $\W (1-Q)$ which is equal to $\beta^{-1}$. This proves \eqref{finlacom}. 

Now let us prove \eqref{juli10}. We must have $Q_1f= u_1 C(f)$ with some constant $C(f)$. From \eqref{finlacom} we have that $\W^rf -u_1 C(f)$ tends to zero when $r\to\infty$, and since $\W$ preserves the integral of $f$ we must have that $C(f)=\int_0^1 f(x)\d x$. 

Point (iii) is just a reformulation of \eqref{finlacom} in the case of simple eigenvalues, and we can identify each $\hat{U}_j$ using \eqref{chiuj}. 

Finally, point (iv) gives a sharper norm estimate of the error without the extra $\alpha$, because in the Parry-Pisot case we do not employ the Gelfand formula. Here we use the first alternative in \eqref{jun1} where we choose the function as $f-u_1 \int_0^1 f(x)\d x$.  \qed

\begin{remark}
    For completeness, let us give an example of a Parry polynomial of degree $4$ with roots outside the closed unit disk. 
Let $L\geq 3$ be an integer and let $p(z)=z^4-\big ((L+1)z^3+Lz^2 +L\big )$ for which $n=4$, $a_0=L+1$, $a_2=0$, $a_1=a_3=L$. Since $a_0$ is strictly larger than all other coefficients, they form a Parry sequence (see \eqref{dec5}). Let 
\begin{equation*}
    f(z)=-\frac{p(z)}{z^3}=L+1+\frac{L}{z} +\frac{L}{z^3} -z,\quad z\neq 0.
\end{equation*}
The zeros of $f$ and $p$ coincide. 
If $x\in \R\setminus\{0\}$ then $f'(x)=-L/x^2 -3L/x^4-1<0$, with $\lim_{x\to-\infty}f(x)=+\infty$  and $\lim_{x\nearrow 0}f(x)=-\infty$. 

We have $f(-1)=-L+2<0$, thus there must exist a unique negative zero $z_2<-1$ of $f$ (hence $p$). From \cite{So} we also know that $|z_2|<\varphi=\frac{\sqrt{5}+1}{2}$. 
\end{remark}

\subsection{Proof of Theorem \ref{thm:left}}\label{proof-left}

{\it Proof of (i)}. 
When $\zeta=1$ then $\hat{U}_m(F)$ is a Riemann sum which converges to the Riemann integral of $F$; note that any { real} function of bounded variation can be written as the difference between two bounded and increasing functions, hence it is Riemann integrable. Thus we may assume that $1/\beta<|\zeta|<1$ or $1<|\lambda|<\beta$. Let us first prove that the limit exists.  Let us focus on the sum over the last indices $j_m$ and $k_m$. We have $k_m\in \{1,2,...,n\}$ and $j_m\in \{0,...,a_{k_m-1}-1\}$. Then we need to control

\begin{equation*}
    \sum_{t_{{\bf k}_{m-1}}^{{\bf j}_{m-1}}} \lambda^{-k_1-k_2-\cdots -k_{m-1}} \sum_{k_m=1}^n \sum_{j_m=0}^{a_{k_m-1}-1} \lambda^{-k_m} f\big (t_{\bk_{m-1}}^{\bj_{m-1}}+\beta^{-|\bk_{m-1}|} t_{k_m}^{j_m}\big ).
\end{equation*}

Here \eqref{aug1} plays a crucial role: the above sum over $k_m$ and $j_m$ can be rewritten as
\begin{equation*}
    f\big (t_{\bk_{m-1}}^{\bj_{m-1}}\big ) +\sum_{k_m=1}^n \sum_{j_m=0}^{a_{k_m-1}-1} \lambda^{-k_m} \Big (f\big (t_{\bk_{m-1}}^{\bj_{m-1}}+\beta^{-|\bk_{m-1}|} t_{k_m}^{j_m}\big ) -f\big (t_{\bk_{m-1}}^{\bj_{m-1}}\big )\Big ).
\end{equation*}
The key estimate is as follows:
\begin{align*}
& \sum_{k_m} \sum_{j_m}|\lambda|^{-k_m}\Big \vert f(t_{\bk_{m-1}}^{\bj_{m-1}}+\beta^{-|\bk_{m-1}|}t_{k_m}^{j_m})-f(t_{\bk_{m-1}}^{\bj_{m-1}})\Big \vert \leq \sum_{k=1}^n\frac{a_{k-1}}{|\lambda|^k}\, \, {\rm Var}^{t_{\bk_{m-1}}^{{\bj_{m-1}}}+\beta^{-|\bk_{m-1}|}}_{t_{\bk_{m-1}}^{\bj_{m-1}}}(f).
\end{align*}
Then when we sum over the partition with $m-1$ layers we have 
\begin{equation*}
    \vert\hat{U}_m(f)-\hat{U}_{m-1}(f)\vert \leq \vert\lambda \vert^{-m+1}\, \sum_{k=1}^n\frac{a_{k-1}}{\vert \lambda \vert^k}\,  {\rm Var}_0^1(f)
\end{equation*}
which defines a convergent series because $|\lambda|>1$. Thus $\hat{U}(f)$ exists. 

Now let us prove \eqref{aug33}. We will show that for all $m\geq n+1$ we have:
\begin{equation}\label{aug34}
    \hat{U}_m(\W f)=\zeta \, \hat{U}_{m-n}(f) +\mathcal{O}(|\lambda|^{-m}).
\end{equation}
From the definition of $\W$ we have 
\begin{equation*}
    \hat{U}_m(\W f)=\beta^{-1}\sum_{s=0}^{a_0-1} \sum_{t_{{\bf k}_m}^{{\bf j}_m}}  \lambda^{-|\bk_m|}f\big (s/\beta +\beta^{-1}t_{{\bf k}_m}^{{\bf j}_m}\big )  +  \beta^{-1}\sum_{t_{{\bf k}_m}^{{\bf j}_m}}  \lambda^{-|\bk_m|}f\big (a_0/\beta +\beta^{-1}t_{{\bf k}_m}^{{\bf j}_m}\big ).
\end{equation*}

The first term contains 
\begin{equation*}
    s/\beta +\beta^{-1}t_{k_1 \cdots k_m}^{j_1 \cdots j_m} =t_{1 k_1 \cdots k_m}^{s j_1\cdots j_m}
\end{equation*}
which will eventually run over all partition elements with ``$m+1$" layers where the value of the first ``$k$" equals $1$ and the first ``$j$" runs from $0$ to $a_0-1$.

The second term contains 
\begin{equation*}
    a_0/\beta +\beta^{-1}t_{{\bf k}_m}^{{\bf j}_m}
\end{equation*}
and we will show that summing over $t_{{\bf k}_m}^{{\bf j}_m}$ in the above formula is the same as summing over all points with $k_1\geq 2$. 

In order to make the main idea of the proof more transparent, we first consider a ``simpler case": 

\subsubsection*{Case 1: $a_0>a_j$ for all $j\geq 1$.} 

We have 
\begin{equation*}
    b_1=a_1/\beta +a_2/\beta^2+\cdots +a_{n-1}/\beta^{n-1}=t_{1\times (n-1)}^{a_1\cdots a_{n-1}}
\end{equation*}
where the notation $1\times(n-1)$ means that $k_1=\cdots=k_{n-1}=1$. Since $f(a_0/\beta+\beta^{-1} t)=0$ if $t>b_1$, it is enough to consider points where $t_{\bk_m}^{\bj_m}<b_1$. This inequality holds if: 
\begin{align*}
&    k_1=1, \quad j_1<a_1\\
 &   k_1=1,\quad j_1=a_1; \quad k_2=1, \quad j_2<a_2 \\ 
  &  k_1=1,\quad j_1=a_1; \quad k_2=1, \quad j_2=a_2; \quad k_3=1,\quad j_3<a_3\\
   & \hspace{0.6cm}\vdots
\end{align*}
In the first case where $k_1=1$ and $j_1<a_1$ we have 
\begin{equation*}
    a_0/\beta +\beta^{-1} (t_1^{j_1} +\beta^{-1}t_{k_2\cdots k_m}^{j_2\cdots j_m})=t_2^{j_1}+\beta^{-2} t_{k_2\cdots k_m}^{j_2\cdots j_m}=t_{2k_2\cdots k_m}^{j_1j_2\cdots j_m}.
\end{equation*}

In the second case where $k_1=k_2=1$, $j_1=a_1<a_0$ and $j_2<a_2$: 
\begin{equation*}
    a_0/\beta +\beta^{-1} (a_1/\beta +\beta^{-1}t_{1k_3\cdots k_m}^{j_2\cdots j_m})=a_0/\beta +a_1/\beta^2 + j_2/\beta^3+\beta^{-3}t_{k_3\cdots k_m}^{j_2\cdots j_m}=t_{3k_2\cdots k_m}^{j_2\cdots j_m}.
\end{equation*}

Now we show that ``no points are missing". Let $p>2$. Then

\begin{equation*}
    t_{2k_2k_3\cdots k_p}^{j_1j_2\cdots j_p}=a_0/\beta +j_1/\beta^2 +\beta^{-2}t_{k_2\cdots k_p}^{j_2\cdots j_p} =a_0/\beta +\beta^{-1} t_{1k_2\cdots k_p}^{j_1j_2\cdots j_p}
\end{equation*}
and if $i\geq 3$ then 
\begin{equation}\label{aug35}
\begin{aligned}
    t_{ik_2k_3\cdots k_{p}}^{j_1j_2\cdots j_p}&=a_0/\beta +a_1/\beta^2+\cdots +a_{i-2}/\beta^{i-1}+j_1/\beta^{-i}+\beta^{-i}t_{k_2\cdots k_p}^{j_2\cdots j_p}\\
    &=a_0/\beta +
    \beta^{-1} t_{1\times (i-1)k_2\cdots k_p}^{a_1\cdots a_{i-2} j_1\cdots j_p},
    \end{aligned}
\end{equation}
where the notation  $1\times (i-1)$ on the right hand side means that the first $i-1$ ``$k$"'s are equal to $1$. 
From \eqref{aug35} we see that given $m$ and $i$, we may replace $a_0/\beta +\beta^{-1} t_{1\times (i-1),k_2\cdots k_m}^{a_1\cdots a_{i-2} j_1\cdots j_m}$ with $t_{ik_2k_3\cdots k_{m-i+2}}^{j_1j_2\cdots j_{m-i+2}}$. Also, we see that the sum of $k$'s on the left hand side of \eqref{aug35} is exactly one unit larger than the sum of $k$'s on the right hand side. This implies that if we want $\hat{U}_m(\W f)$ to resemble $\hat{U}_m(f)$, we need to multiply it with $\beta/\lambda$. 

Thus 
\begin{align}\label{aug355}
& (\beta/\lambda)\hat{U}_m(\W f)=\sum_{t_{1k_1\cdots k_m}^{sj_1\cdots j_{m}}} \lambda^{-1-|\bk_m|} f(t_{1k_1\cdots k_m}^{sj_1\cdots j_{m}})+\sum_{t_{2k_2\cdots k_m}^{j_1\cdots j_{m}}} \lambda^{-2-k_2-\cdots -k_m} f(t_{2k_2\cdots k_m}^{j_1\cdots j_{m}})\\
&+\sum_{t_{3k_2\cdots k_{m-1}}^{j_1\cdots j_{m-1}}} \lambda^{-3-k_2-\cdots -k_{m-1}} f(t_{3k_2\cdots k_{m-1}}^{j_1\cdots j_{m-1}})+\cdots +\sum_{t_{nk_2\cdots k_{m-n+2}}^{j_1\cdots j_{m-n+2}}} \lambda^{-n-k_2-\cdots -k_{m-n+2}} f(t_{nk_2\cdots k_{m-n+2}}^{j_1\cdots j_{m-n+2}}).\nonumber 
\end{align}
Each sum on the right hand side is performed over a different number of layers varying from $m-n+1$ to $m+1$. Since $f$ has bounded variation, we may discard a finite number of layers with a price of an error $\mathcal{O}(|\lambda|^{-m})$ and this proves \eqref{aug34}. Then take $m\to\infty$ and obtain \eqref{aug33} when $a_0>a_j$ for all $1\leq j\leq n-1$. 

\subsubsection*{Case 2: general Parry coefficients.}

 In Appendix \ref{sec-relab} we will prove the following result: 
\begin{lemma}\label{lemma-relabeling} Any number $t_{k_1k_2\cdots k_m}^{j_1j_2\cdots j_m}$ with $k_1\geq 2$ can be written as 
\begin{equation*}
    a_0/\beta +\beta^{-1}\tilde{t}\quad \text{with}\quad \tilde{t}= t_{\tilde{k}_1\cdots \tilde{k}_p}^{\tilde{j}_1\cdots \tilde{j}_p},
\end{equation*}
where 
\begin{equation*}
    k_1+k_2+\cdots +k_m=1 +\tilde{k}_1+\cdots +\tilde{k}_p,\quad m\leq p\leq n-1+m.
\end{equation*}

The other way around, let $t_{{k}_1\cdots {k}_m}^{{j}_1\cdots {j}_m}<b_1=a_1/\beta+\cdots +a_{n-1}/\beta^{n-1}$. Then there exists $p\geq 1$ and $t_{\tilde{\bk}_p}^{\tilde{\bj}_p}<1$ with $\tilde{k}_1\geq 2$ such that 
\begin{equation*}
    t_{\tilde{\bk}_p}^{\tilde{\bj}_p}=a_0/\beta +\beta^{-1} t_{\bk_m}^{\bj_m}.
\end{equation*}
\end{lemma}

This technical result replaces the identity \eqref{aug35}, which was the crucial ingredient in proving \eqref{aug355}. This time we cannot exactly say how many layers we loose or gain by replacing the summation over the indices $\tilde{k}$ and $\tilde{j}$ with indices where the first ``$k$" equals $2,3,...,n$. But \eqref{aug34} still holds, and this finishes the proof of (i).  

{\it Proof of (ii)}.   Let us introduce a modified version of $\hat{U}_m(F)$ given by 
 \begin{equation*}
     \widetilde{U}_m(f )=\sum_{t_{\bk_m}^{\bj_m}}\lambda^{-|\bk_m|} f\big (t_{\bk_m}^{\bj_m}+\beta^{-|\bk_m|}\big ).
 \end{equation*}
 
It turns out that one can explicitly compute $\widetilde{U}_m(f)$ when $f=\chi_{[0,y]}$ with $0<y<1$. Let us do that. 
 There exists a unique sequence $t_{u_1u_2\cdots u_p}^{w_1w_2\cdots w_p}$ such that 
\begin{equation*}
    t_{u_1u_2\cdots u_p}^{w_1w_2\cdots w_p}\leq \,  y=\sum_{r\geq 1} \frac{y_{r-1}}{\beta^r}\, <  t_{u_1u_2\cdots u_p}^{w_1w_2\cdots w_p}+\beta^{-u_1-\cdots -u_p},\quad p\geq 1.
\end{equation*}
For example 
\begin{equation*}
    t_{u_1}^{w_1}=a_0/\beta+\cdots +a_{u_1-2}/\beta^{u_1-1} +w_1/\beta^{u_1},\quad w_1=y_{u_1-1},
\end{equation*}
where by convention, if $u_1=1$, then only the term $y_0/\beta$ appears.  If $u_1>1$ then $y_0=a_0$, $y_1=a_1$, ..., $y_{u_1-2}=a_{u_1-2}$. 

For $m=1$ we have 
\begin{equation*}
    \widetilde{U}_1\big (\chi_{[0,y]}\big )=\sum_{k_1=1}^n \lambda^{-k_1}\sum_{j_1=0}^{a_{k_1-1}-1} \chi_{[0,y]}(t_{k_1}^{j_1}+\beta^{-k_1})=\sum_{r=1}^{u_1}y_{r-1}/\beta^r
\end{equation*}
where we see that when $k_1=u_1$, only the values $j_1=0,1,\ldots ,w_1-1$ will produce a value $1$, hence by summing them we obtain $w_1=y_{u_1-1}$. 

For $m=2$ we have 
\begin{align*}
   \widetilde{U}_2\big (\chi_{[0,y]}\big ) =\sum_{k_2=1}^n \sum_{j_2=0}^{a_{k_2-1}-1} \lambda^{-k_2}\sum_{k_1=1}^n\lambda^{-k_1}\sum_{j_1=0}^{a_{k_1-1}-1} \chi_{[0,y]}\big (t_{k_1}^{j_1}+\beta^{-k_1}t_{k_2}^{j_2} +\beta^{-k_1-k_2}\big ).
\end{align*}
If $k_1<u_1$ or if $k_1=u_1$ and $j_1<w_1$ then the value of the function is $1$, regardless of $k_2$ and $j_2$. This part reproduces $\tilde{U}_1$. If $k_1=u_1$ and $j_1=w_1$ then the summation over $k_2$ and $j_2$ is no longer trivial and we have the following extra terms: 
\begin{align*}
    \sum_{k_2=1}^n \sum_{j_2=0}^{a_{k_2-1}-1} \lambda^{-u_1-k_2} \chi_{[0,y]}\big (t_{u_1}^{w_1}+\beta^{-u_1}t_{k_2}^{j_2} +\beta^{-u_1-k_2}\big )=\sum_{r=u_1+1}^{u_1+u_2}\frac{y_{r-1}}{\lambda^r}.
\end{align*}
In the same way we may prove 
\begin{equation*}
    \widetilde{U}_m(\chi_{[0,y]})=\sum_{r=1}^{u_1+\cdots+u_m}\frac{y_{r-1}}{\lambda^r}\quad m\geq 1.
\end{equation*}
The final step is to take $m\to\infty$. 

In order to prove \eqref{aug31} we use
\begin{align*}
h(1) - h(t) &= \int_t^1 h'(y)dy=\int _0^1 \chi_{[0,y]}(t) h'(y)dy
\end{align*}
and apply the formula for $\hat{U}(\chi_{[0,y]})$. 

Now let us prove \eqref{sep11}. 
Using \eqref{marvel} we have 
\begin{align*}
   \hat{U}(u)=\sum_{k=0}^{n-1}\lambda^{-k-1}\hat{U}(\chi_{[0,b_k]})= \sum_{k=0}^{n-1} \sum_{j\geq 1} \lambda^{-k-j -1} a_{k+j-1}=\sum_{p\geq 1} p \, \frac{a_{p-1}}{\lambda^{p+1}}=h_n'(\lambda).
\end{align*}
\qed 

{\it Proof of (iii)}. Denote by $f(y)=\hat{U}(\chi_{[0,y]})$. 
Assume first that 
\begin{equation*}
    y=\sum_{m\geq 1}\frac{y_{k_m}}{\beta^{k_m}},\quad y_{k_m}>0.
\end{equation*}

Let $Y_M=\frac{y_{k_M}-1}{\beta^{k_M}} +\sum_{m\neq M} \frac{y_{k_m}}{\beta^{k_m}}$. Then $f(y)-f(Y_M)=\lambda^{-k_M}$ and $y-Y_M=\beta^{-k_M}$ for all $M$ and $f'(y)$ does not exist. 

Now assume that $y=\sum_{k=1}^K \frac{y_k}{\beta^k}$ with $y_K>0$. For every $m>K+2$ we define $Y_m=y+\beta^{-m}<1$. Then $f(Y_m)-f(y)=\lambda^{-m}$ and $Y_m-y=\beta^{-m}$ for all $m$ and $f'(y)$ does not exist.

{\it Proof of (iv)}.  The proof is identical with the one in  \cite[Theorem 4.5]{Su}. Assume that such a $\mu$ exists. Then the map from (iii)  would equal $\mu([0,y])$. This is a function of bounded variation, thus differentiable almost everywhere, which is in contradiction with (iii). 

{\it Proof of (v)}. If $q$ is a polynomial of degree $M$ then by choosing $m=1$ and $N=M+1$ in formula \eqref{cdh20} we see that the error term is identically zero and for $r\geq n$ we have 
\begin{align}\label{cdh290}
   & \big(\mathscr{P}^r q\big)(y) = \sum_{k_1=1}^n \sum_{j_1=0}^{a_{k_1-1}-1}\bigg\{  \big(\mathscr{P}^{r-k_1}H_1\big)(y) \int_{t_{k_1}^{j_1}}^{t_{k_1}^{j_1}+\beta^{-k_1}} q(u) \d u  \\\nonumber 
        &+ \sum_{s=1}^{M+1} \beta^{-k_1 s}\Big( q^{(s-1)}\big(t_{k_1}^{j_1}+\beta^{-k_1}\big)-q^{(s-1)}\big(t_{k_1}^{j_1}\big)\Big)  \frac{1}{s!}\Vert B_s \Vert_{L^1([0,1])} \big(\mathscr{P}^{r-k_1} H_{ns+1}\big)(y)\bigg \}.
\end{align}
Note that the term with $s=M+1$ is zero because $q^{(M)}$ is constant. 

Let $1/\beta<|\zeta|<1$ be a simple discrete eigenvalue. It will also be an eigenvalue for the matrix $\widetilde{\mathcal{P}}_{\nu,n}$ from \eqref{cdh30}, see also the invariance property in \eqref{cdh26}. Its Riesz projection $\pi_\zeta$ will have rank one and
\begin{equation*}
    \pi_\zeta (H_{ns+1})=C_s\, u_\zeta,\quad 0\leq s\leq M,
\end{equation*}
where $u_\zeta$ is the eigenvector from \eqref{sep10} with $c_1=1/\lambda$ for which we know the value of $\hat{U}(u_\zeta)$ from \eqref{sep11}. We also have $C_0=1/\hat{U}(u_\zeta)$ because $\pi_\zeta(\chi_{[0,1]})=u_\zeta \frac{1}{\hat{U}(u_\zeta)} \hat{U}(\chi_{[0,1]})$ and $\hat{U}(\chi_{[0,1]})=1$. 

 $\hat{U}$ applied on any other (generalized) eigenvector of $\W$ (in $BV([0,1])$) or on vectors in the range of the Riesz projection corresponding to the essential spectrum of $\W$ gives zero, so applying $\hat{U}$
on both sides of \eqref{cdh290} we obtain 
\begin{align*}
   & \zeta^r \hat{U}(q) = \sum_{k_1=1}^n \sum_{j_1=0}^{a_{k_1-1}-1}\, \zeta^{r-k_1}\bigg\{   \int_{t_{k_1}^{j_1}}^{t_{k_1}^{j_1}+\beta^{-k_1}} q(u) \d u  \\\nonumber 
        &+ \hat{U}(u_\zeta)\sum_{s=1}^{M+1} C_s \beta^{-k_1 s}\Big( q^{(s-1)}\big(t_{k_1}^{j_1}+\beta^{-k_1}\big)-q^{(s-1)}\big(t_{k_1}^{j_1}\big)\Big)  \frac{1}{s!}\Vert B_s \Vert_{L^1([0,1])} \bigg \}.
\end{align*}
Dividing by $\zeta^r$ we finally obtain \eqref{cdh291}. 
When $q(x)=1$, the only non-zero term comes from the integral and we can see that $\hat{U}(\chi_{[0,1]})=1$, as expected.
\qed 

\begin{appendix}\appheader
\section{The invariant subspace}\label{sec:invsubspace} 
We recall the definition of the following linearly independent  (due to Lemma \ref{lemmaParry}) characteristic functions
\begin{align*}
    \psi_1(x) = \chi_{[0,1]}(x), \quad   \psi_2(x) = \chi_{[0,b_1]}(x), \quad  \ldots, \quad
      \psi_{n}(x) = \chi_{[0,b_{n-1}]}(x).
\end{align*}

Applying the transfer operator $\mathscr{P}$ to each of these functions gives for $1\leq j\leq n-1$ and for Lebesgue almost every $x$ (see also \cite{CS}):
\begin{align*}
    \beta ( \mathscr{P}\psi_j)(x) 
    &=a_{j-1} \psi_1(x)  +\psi_{j+1}(x),
\end{align*}
and
\begin{align*}
    \beta ( \mathscr{P}\psi_n)(x) 
    &=a_{n-1} \psi_1(x).
\end{align*}
This shows that the functions $\psi_1,\ldots,\psi_n$ span  an invariant subspace for $\beta \mathscr{P}$ with a restriction matrix equal to
\begin{align}\label{Pn}
\mathcal{P}_n=    \begin{pmatrix}
        a_0 & a_1 & \cdots & \cdots & a_{j-1} & \cdots & a_{n-2} &  a_{n-1} \\
        1 & 0 &  & \cdots & 0 & \cdots & & 0 \\
        0 & 1 &  & & 0 & & 0& 0 \\
        \vdots & \vdots & \ddots & & 0 & &\vdots & \vdots \\
        0 & 0 & 0 & \ddots & 0 &  & 0 & 0 \\
        0 & 0 & 0 &  & 1 &  & 0 & 0 \\
        \vdots & \vdots & &  & 0 & \ddots & \vdots & \vdots \\
        0 & 0 & 0 & 0 &0 & \cdots & 1 & 0 
    \end{pmatrix}.
\end{align}
Let $\psi_{\ell,j}(x) = x^\ell\psi_j(x)$ with $1\leq j\leq n$ and $\ell\geq 1$. Then using the binomial formula and the previous formulas we have ($j\leq n-1$):
\begin{align*}
    \beta ( \mathscr{P}\psi_{\ell,j})(x) &= \frac{a_{j-1}x^\ell}{\beta^\ell} \psi_1(x) + \frac{x^\ell}{\beta^\ell}\psi_{j+1}(x) + R_{\ell-1,j}(x)\\
    &=\frac{a_{j-1}}{\beta^\ell} \psi_{\ell, 1}(x) + \frac{1}{\beta^\ell}\psi_{\ell,j+1}(x) + \sum_{k=0}^{\ell -1} \sum_{m=1}^n C_{\ell,j, k,m} \, \psi_{k,m}.
\end{align*}
where  $R_{\ell-1,j}$ are remaining terms of degree at most $\ell-1$ in $x$. Similarly: 
\begin{align}
    \beta ( \mathscr{P}\psi_{\ell,n})(x) =\frac{a_{n-1}}{\beta^\ell} \psi_{\ell, 1}(x) + \sum_{k=0}^{\ell -1}\sum_{m=1}^n C_{\ell, k,m} \, \psi_{k,m}.
\end{align}
Now let us renumber the functions $\psi_{\ell,j}(x)=x^\ell \psi_j(x)$ with $0\leq \ell \leq \nu-1$ and $1\leq j\leq n$ as 
\begin{align}\label{eq:ks10}
    \Psi_1=\psi_{0,1},\, \ldots , \, \Psi_n=\psi_{0,n},\, \Psi_{n+1}=\psi_{1,1},\, \ldots,\,  \Psi_{2n}=\psi_{1,n},\, \ldots,\, \Psi_{\nu n}=\psi_{\nu,n}.
\end{align}
We see that the restriction matrix $\mathcal{P}_{\nu,n}$ to this $\nu n$ dimensional subspace generated by the functions $\Psi_{k}$ with $1\leq k \leq \nu n$ has $\nu$ blocks on the ``diagonal" which are all of the type $n\times n$ and equal to $\beta^{-k } \mathcal{P}_n$ with $0\leq k\leq \nu-1$:  
\begin{align}\label{cdh23}
\mathcal{P}_{\nu,n}=    \begin{pmatrix}
        \mathcal{P}_n & \star & \star & \cdots & \star \\
        {\bf 0}_n & \beta^{-1}\mathcal{P}_n& \star & \cdots & \star \\
        {\bf 0}_n & {\bf 0}_n & \beta^{-2}\mathcal{P}_n & \cdots & \star\\
        \vdots & \vdots &\quad \ddots & & \star \\
        {\bf 0}_n & {\bf 0}_n & {\bf 0}_n & \ddots & \star  \\
        {\bf 0}_n & {\bf 0}_n & {\bf 0}_n & \cdots & \beta^{-(\nu-1)}\mathcal{P}_n
    \end{pmatrix}.
    \end{align}
This implies that the spectrum of $\mathcal{P}_{\nu,n}$ is given by 
\begin{align}\label{eq:eigenvaluesPnun}
\sigma(\mathcal{P}_{\nu,n})=\bigcup_{\ell=0}^{\nu -1} \beta^{-\ell} \sigma\big (\mathcal{P}_{ n}\big ).
\end{align}

\begin{lemma}\label{lemmahc1}
We have that $\sigma(\mathcal{P}_n)=\{\beta,\lambda_2,...,\lambda_n\}$ where $\beta>1$ is a simple eigenvalue and $0<|\lambda_j|<\beta$ for all $2\leq j\leq n$. Moreover, the geometric multiplicity of all eigenvalues is always equal to $1$. 
\end{lemma}
\begin{proof} Another way of formulating the lemma is that $\beta$ is a simple zero for ${\rm Det}(z\bb1_n -\mathcal{P}_n)$, and all  other zeros of this polynomial have a length less than $\beta$. 
   We have 
   \begin{equation*}
       {\rm Det}(z\bb1_n -\mathcal{P}_n)=(z-a_0)z^{n-1}-a_1z^{n-2}-\cdots-a_{n-1}=z^n-(a_0z^{n-1}+a_1z^{n-2} +\cdots +a_{n-1}).
   \end{equation*}
   Then $z$ is a zero for the determinant iff $z$ is a zero for the function 
   \begin{equation}\label{pt1}
    f(z)=1-(a_0z^{-1}+a_1z^{-2} +\cdots +a_{n-1}z^{-n}).
    \end{equation}
    {We already know that $\beta>a_0$ is the unique positive solution of $f(x)=0$ with $x>0$. Since 
     ${\rm Det}(z\bb1_n -\mathcal{P}_n)=z^nf(z)$ and $f'(\beta)\neq 0$ we conclude that $\beta$ is a simple root. 
     
    Next, let us assume that $|z|>\beta$. Then 
    \begin{equation*}
        \vert f(z) \vert \geq 1-\sum_{j=0}^{n-1} a_j \vert z \vert^{-j-1}>f(\beta)=0.
    \end{equation*}
    This implies that the zeros of $f$ have a length at most $\beta$.
    
    Now let us assume that there exists $\phi\in (0,\pi]$ such that $\rho=\beta e^{i\phi}$ is a root. We have 
    \begin{equation*}
        0=f(\beta e^{i\phi})=1- \sum_{j=0}^{n-1} a_j \beta^{-j-1} e^{-i (j+1)\phi}.
    \end{equation*}
Taking the real part and using \eqref{eq:fundeqbeta} we have
\begin{align*}
0&=1-\sum_{j\geq 0} a_j \beta^{-j-1} \cos( (j+1)\phi)=\sum_{j\geq 0} a_j \beta^{-j-1} \big (1-\cos( (j+1)\phi)\big )\\
&\geq a_0\beta^{-1} (1-\cos(\phi))
\end{align*}
which leads to $\phi=0$, contradiction. This implies that any zero of $f$ other than $\beta$ must also have a length less than $\beta$. }

{ Now let us show that the geometric multiplicity of all eigenvalues equals one. We have 
\begin{equation}\label{aug70}
\begin{aligned}
    \beta (\mathscr{P}\psi_j)(x) &= a_{j-1}\psi_1(x) + \psi_{j+1}(x) , \qquad \textup{for } j \in \{1,\ldots,n-1\}, \\
    \beta (\mathscr{P}\psi_n)(x) &= a_{n-1}\psi_1(x).
\end{aligned}
\end{equation}
Any eigenfunction corresponding to an eigenvalue $\lambda$ is given by a linear combination 
\begin{align*}
    u(x) = c_1 \psi_1(x) + c_2 \psi_2 + \cdots + c_n \psi_n(x).
\end{align*}
Then $\beta \W u=\lambda u$, \eqref{aug70} and $1=\sum_{k=1}^n\frac{a_{k-1}}{\lambda^k}$ lead to 
\begin{equation}\label{sep10}
u=c_1 \bigg( \psi_1(x) + \frac{1}{\lambda} \psi_2(x) + \cdots + \frac{1}{\lambda^{n-1}}\psi_n(x)\bigg)=c_1\sum_{k=0}^{n-1}\lambda^{-k} \chi_{[0,b_k]},
\end{equation}
which shows that this is a one-dimensional subspace. When $\lambda=\beta>1$ we may choose  a unique $c_1 = C_0$, such that
\begin{align}\label{cdh1}
    u_1(x) = C_0 \bigg( \psi_1(x) + \frac{1}{\beta} \psi_2(x) + \cdots + \frac{1}{\beta^{n-1}}\psi_n(x)\bigg) \quad \text{and}\quad \int_0^1 u_1(x) \d x=1.
\end{align}
Let us also compute a left eigenvector $w_1=[d_1,d_2,\ldots ,d_n]$ such that $w_1\mathcal{P}_n=\beta w_1$. Imposing $d_1=1$ we  have the equations 
\begin{equation*}
    d_1=1,\quad d_2=\beta-a_0,\quad d_3=\beta d_2-a_1=\beta^2-a_0\beta -a_1,\, \ldots .
\end{equation*}
Comparing with \eqref{dec1} we see that 
$w_1=[1,b_1,b_2,\ldots, b_{n-1}].$ We see that the constant $C_0$ from \eqref{cdh1} obeys 
\begin{equation*}
    C_0\bigg( 1 + \frac{b_1}{\beta} + \cdots + \frac{b_{n-1}}{\beta^{n-1}}\bigg) =1,
\end{equation*}
which implies that the rank one Riesz projection $\Pi_\beta$ corresponding to $\beta$ acts on a given basis element $\psi_k(x)$ as  
$(\Pi_\beta \psi_k)(x)=b_{k-1} \, u_1(x)$, $1\leq k\leq n.$ In particular:   
\begin{equation}\label{mars1}
    (\Pi_\beta \psi_1)(x)= u_1(x).  
\end{equation}
}
\end{proof}

\begin{remark}
    There exist Parry polynomials which have eigenvalues with algebraic multiplicity larger than one. For example: 
    \begin{equation*}
        p(z)=z^7-5z^6-5z^5-z^4-5z^3-z^2-3z -5,
    \end{equation*}
where  $a_0=5$, $a_1=5$, $a_2=1$, $a_3=5$, $a_4=1$, $a_5=3$ and $a_6=5$. We see that the Parry conditions \eqref{Parrycon} are satisfied, and at the same time: 
\begin{equation*}
    p(z)=(z + 1)^2 (z^5 - 7 z^4 + 8 z^3 - 10 z^2 + 7 z - 5),
\end{equation*}
which shows that $-1$ has algebraic multiplicity 2. Since we know from Lemma \ref{lemmahc1} that the geometric multiplicity is always 1, any possible multiple eigenvalue will have a non-trivial Jordan block. 
\end{remark}

 We also need to investigate how the restriction matrix looks. when we use the basis consisting of scaled Bernoulli polynomials $H_{sn+j}$ from Definition~\ref{Giz}. For $s \in \{0,\ldots,\nu-1\}$, $j \in \{1,\ldots,n\}$ and $k \in \{1,\ldots,\nu n\}$ we have 
\begin{align}\label{cdh26}
    \big(\mathscr{P}H_{sn+j}\big)(x) =\sum_{k=1}^{\nu n} \widetilde{p}_{k,sn+j} H_k(x),
\end{align}
where $\widetilde{\mathcal{P}}_{\nu,n}= \big\{\widetilde{p}_{k,sn+j}\big\}$ is a matrix conjugated to $\beta^{-1}\, \mathcal{P}_{\nu,n}$. We note that the degree of the polynomial in $H_{sn+j}$ coincides with the degree of $\Psi_{sn+j}=x^s\psi_j(x)$'s from \eqref{eq:ks10}, which implies that the $\nu n \times \nu n$ ``change of basis" invertible matrix $S$ which transforms the $\Psi$'s into $H$'s has the same block structure as $\mathcal{P}_{\nu,n}$, that is 
\begin{align*}
S_{\nu,n}=    \begin{pmatrix}
        \mathcal{S}_{n,1} & \star & \star & \cdots & \star \\
        {\bf 0}_n & \mathcal{S}_{n,2}& \star & \cdots & \star \\
        {\bf 0}_n & {\bf 0}_n & \mathcal{S}_{n,3} & \cdots & \star\\
        \vdots & \vdots &\quad \ddots & & \star \\
        {\bf 0}_n & {\bf 0}_n & {\bf 0}_n & \ddots & \star  \\
        {\bf 0}_n & {\bf 0}_n & {\bf 0}_n & \cdots & \mathcal{S}_{n,\nu}
    \end{pmatrix},\quad S^{-1}_{\nu,n}=    \begin{pmatrix}
        \mathcal{S}^{-1}_{n,1} & \star & \star & \cdots & \star \\
        {\bf 0}_n & \mathcal{S}^{-1}_{n,2}& \star & \cdots & \star \\
        {\bf 0}_n & {\bf 0}_n & \mathcal{S}^{-1}_{n,3} & \cdots & \star\\
        \vdots & \vdots &\quad \ddots & & \star \\
        {\bf 0}_n & {\bf 0}_n & {\bf 0}_n & \ddots & \star  \\
        {\bf 0}_n & {\bf 0}_n & {\bf 0}_n & \cdots & \mathcal{S}^{-1}_{n,\nu}
    \end{pmatrix}.
\end{align*}
Thus $\widetilde{\mathcal{P}}_{\nu,n}=S_{\nu, n} \big (\beta^{-1}\mathcal{P}_{\nu, n}\big ) S^{-1}_{\nu,n}$ still has a block structure of the type

\begin{align}\label{cdh30}
\widetilde{\mathcal{P}}_{\nu,n}=    \begin{pmatrix}
        \mathcal{A}_{n,1} & \star & \star & \cdots & \star \\
        {\bf 0}_n & \mathcal{A}_{n,2}& \star & \cdots & \star \\
        {\bf 0}_n & {\bf 0}_n & \mathcal{A}_{n,3} & \cdots & \star\\
        \vdots & \vdots &\quad \ddots & & \star \\
        {\bf 0}_n & {\bf 0}_n & {\bf 0}_n & \ddots & \star  \\
        {\bf 0}_n & {\bf 0}_n & {\bf 0}_n & \cdots & \mathcal{A}_{n,\nu}
    \end{pmatrix}
\end{align}
where $\mathcal{A}_{n,\ell}=\mathcal{S}_{n,\ell} \big (\beta^{-\ell}\mathcal{P}_n\big ) \mathcal{S}^{-1}_{n,\ell}$ for all $1\leq \ell\leq \nu$. This implies that the spectrum of $\widetilde{\mathcal{P}}_{\nu,n}$ is given by 
\begin{align}\label{cdh31}
\sigma(\widetilde{\mathcal{P}}_{\nu,n})=\bigcup_{\ell=1}^{\nu} \beta^{-\ell} \sigma\big (\mathcal{P}_{ n}\big ).
\end{align}
Now assume that $\widetilde{\mathcal{P}}_{\nu,n}$ has exactly $2\leq E\leq \nu n$ distinct eigenvalues and let $\pi^{(a)}$ where $1\leq a\leq E$ denote the Riesz projection onto the invariant subspace of each distinct eigenvalue $\zeta_a$.  We have:   
\begin{align*}
  \sum_{a=1}^E \pi^{(a)} = I_{\nu n \times \nu n} \quad \textup{ and } \quad  \widetilde{\mathcal{P}}_{\nu,n} = \pi^{(1)} + \sum_{a=2}^E \pi^{(a)} \widetilde{\mathcal{P}}_{\nu,n} \pi^{(a)},
\end{align*}
where we recall that the projection $\pi^{(1)}$ is rank one since $1$ is a simple eigenvalue. Furthermore, if $1\leq m_a$ denotes the algebraic multiplicity of $\zeta_a$ for $a\geq 2$ there exists $C>0$ such that for all $r\geq 1$ we have
\begin{align*}
    \big\Vert (\pi^{(a)}  \widetilde{\mathcal{P}}_{\nu,n} \pi^{(a)})^r \big\Vert \leq C r^{m_a-1}\vert\zeta_a\vert^r.
\end{align*}
{ 
\begin{remark} By direct computation we may derive the formula of the first block $\mathcal{A}_{n,1}$ which equals (see also \eqref{dec1} for the formulas giving the $b$'s):
\begin{align*}
\mathcal{A}_{n,1}=   \beta^{-1} \begin{pmatrix}
          a_0 &   b_1^{-1} a_1 & \cdots & \cdots &   b_{j-1}^{-1} a_{j-1} & \cdots &  b_{n-2}^{-1} a_{n-2} &   b_{n-1}^{-1} a_{n-1} \\
        b_1 & 0 &  & \cdots & 0 & \cdots & & 0 \\
        0 & b_1^{-1} b_2 &  & & 0 & & 0& 0 \\
        \vdots & \vdots & \ddots & & 0 & &\vdots & \vdots \\
        0 & 0 & 0 & \ddots & 0 &  & 0 & 0 \\
        0 & 0 & 0 &  & b_{j-1}^{-1}b_j &  & 0 & 0 \\
        \vdots & \vdots & &  & 0 & \ddots & \vdots & \vdots \\
        0 & 0 & 0 & 0 &0 & \cdots & b_{n-2}^{-1}b_{n-1} & 0 
    \end{pmatrix}.
\end{align*}
The only non-zero term in the last column equals $1$ due to the identity $a_{n-1}/\beta =b_{n-1}$. Also, using the identity $\beta b_{j-1}-a_{j-1}= b_{j}$ we obtain $\beta^{-1}\big (b_{j-1}^{-1} a_{j-1}+b_{j-1}^{-1}b_j\big )=1$ for all $1\leq j\leq n-1$. This shows that $\mathcal{A}_{n,1}$ is a left-stochastic matrix.
\end{remark}
}

We continue this appendix with several important results regarding the transfer operator.
\begin{lemma}\label{lem:ks11}
    Let $N \in \mathbb{N}$ and $f \in C^N([0,1])$. Let $\vert \lambda_*\vert  <1$ be the length of the second largest root (in absolute value) of $p(z)$. Then:
    \begin{enumerate}
        \item $\int_0^1 \big(\mathscr{P}f\big)(x) \d x = \int_0^1 f(x) \d x$,
        \item $\Vert \mathscr{P}f\Vert_{L^1}\leq \Vert f\Vert_{L^1}$,
        \item $\Vert \mathscr{P} f \Vert_{L^{\infty}} \leq \frac{a_0+1}{\beta} \Vert f \Vert_{L^\infty}$,
        \item $\big(\mathscr{P}^r H_{1}\big) (x) = u_1(x) + \mathcal{O}\big (\frac{\vert \lambda_* \vert^r}{\beta^r} \big )$,
  \item  $\big(\mathscr{P}^r H_{sn+1}\big) (x) = \mathcal{O} (\beta^{-r})$ when $s \geq 1$.
    \end{enumerate}
\end{lemma}
\begin{proof}
    \begin{enumerate}
        \item \hspace{-0.2cm} - (iii) They follow by straightforward calculations. 
        \setcounter{enumi}{3}
        \item {We have  $H_1=\psi_1=\chi_{[0,1]}$ and we know that the subspace generated by the characteristic functions $\{\psi_j\}_{j=1}^n$ is left invariant by $\mathscr{P}$.  From \eqref{mars1} we see that $u_1$ is constructed from ``the first column" of the spectral projection corresponding to the eigenvalue $1$ of the matrix $\beta^{-1}\mathcal{P}_n$, hence the function $H_1 -u_1$ lies in the subspace corresponding to the other eigenvalues of $\beta^{-1}\mathcal{P}_n$ which are simple (see Lemma \ref{simple-eig}) and bounded by $|\lambda_*|/\beta$, hence $\mathscr{P}^r(H_1-u_1)$ will obey the desired estimate. 
        
        \item The eigenvalues of all ``diagonal blocks" $\mathcal{A}_{n,\ell}$ with $\ell\geq 2$ in \eqref{cdh30} are in absolute value less than or equal to $\beta^{-1}$. We can write $H_{sn+1}(x)=C_1 u_1(x) + \widetilde{\psi}_1(x)$ where $\widetilde{\psi}_1(x)$ belongs to the subspace generated by the range of all other projections $\pi^{(a)}$ with $a\geq 2$. For $r\geq 2$ we have 
        \begin{equation*}
            \mathscr{P}^rH_{sn+1} =C_1 \, u_1 +\mathcal{O}(\beta^{-r})
        \end{equation*}
        and by integrating on both sides, and using (i) together with the fact that  $\int_0^1H_{sn+1}(x)\d x =0$ we get $C_1=0$.  
        }
    \end{enumerate}
\end{proof}

{
For the special case when the polynomial is Parry-Pisot, we recall from Definition~\ref{def1} that all eigenvalues other than $\beta$ lie in the open unit disk and we thus get a more refined result than above when $s=1$: 

\begin{lemma}\label{lemma26}
   Let $\vert \lambda_*\vert  <1$ be the length of the second largest root (in absolute value) of $p(z)$.  Then the second largest eigenvalue (in absolute value) of $\widetilde{\mathcal{P}}_{2,n}$ equals $\beta^{-1}$, is simple, and there exists a corresponding eigenfunction $v_1(x)\neq 0$ of $\mathscr{P}$ consisting of a linear combination of piecewise polynomials, such that for all $j\geq 1$ we have 
   \begin{align}\label{cdh10}
       \mathscr{P}^j H_{n+1}=\beta^{-j} v_1 +\mathcal{O}\Big (\big (|\lambda_*|/\beta\big )^j\Big ),\quad \int_0^1 v_1(x) \d x=0. 
   \end{align}
\end{lemma}
\begin{proof}
    The second largest in absolute value of the eigenvalues of the first $n\times n$ block $\mathcal{A}_{n,1}$ has a length of $|\lambda_*|/\beta <\beta^{-1}$.  Also, $\W^j H_{n+1}$ stays in the subspace corresponding to the first two blocks of the matrix in \eqref{cdh30}. The second block  $\mathcal{A}_{n,2}$ is conjugate to and has the same spectrum as $\beta^{-2}\mathcal{P}_n$ or $\beta^{-1}\mathcal{A}_{n,1}$. For these two blocks taken together, the second largest eigenvalue becomes $1/\beta$ while $\lambda_*/\beta$ becomes the third largest. The Riesz projection of the eigenvalue $1/\beta$ coming from the second block has rank one, thus when the projection acts on $H_{n+1}$, the result is a function $v_1$ which is a linear combination of piecewise polynomials of degree at most one. 

    Since $\int_0^1 v_1(x) \d x=\int_0^1(\mathscr{P}v_1)(x) \d x=\beta^{-1}\int_0^1 v_1(x) \d x$, we conclude that $\int_0^1 v_1(x) \d x=0$.   Now we may write 
\begin{equation*}
    H_{n+1}(x)= C_1 u_1(x) +v_1(x)+ \widetilde{\psi}_{n+1}(x)
\end{equation*}
where $\mathscr{P}u_1=u_1$, $\mathscr{P}v_1=\beta^{-1} v_1$, and $\widetilde{\psi}_{n+1}$ is in the range of the projections corresponding to eigenvalues of $\widetilde{\mathcal{P}}_{2,n}$ with a length at most equal to $|\lambda_*|/\beta<1/\beta$. The coefficient $C_1$ must be zero because the integrals of $H_{n+1}$ and $v_1$ are zero, where again we use that $\mathscr{P}^j$ preserves the value of the integral.  This leads to \eqref{cdh10}.
\end{proof}

}

\section{Splitting  the interval $[0,1]$}\label{sec:splittingofinterval}
We recall that the first splitting of the interval $[0,1]$ is given by
\begin{align*}
    t_{k_1}^{j_1} = 
    \begin{cases}
\frac{j_1}{\beta} & 0 \leq j_1 \leq a_{0}-1 \qquad \textup{when } k_1 = 1 \\
\frac{a_0}{\beta} + \frac{j_1}{\beta^2} & 0 \leq j_1 \leq a_{1}-1 \qquad \textup{when } k_1 = 2 \\
\frac{a_0}{\beta} + \frac{a_1}{\beta^2} + \frac{j_1}{\beta^3} & 0 \leq j_1 
\leq a_{2}-1 \qquad \textup{when } k_1 = 3 \\
\hspace{1cm}\vdots & \\
\frac{a_0}{\beta} + \frac{a_1}{\beta^2} + \cdots + \frac{a_{n-2}}{\beta^{n-1}} + \frac{j_1}{\beta^n} & 0 \leq j_1 \leq a_{n-1}-1 \qquad \textup{when } k_1 = n
    \end{cases}
\end{align*}
where we use the convention that if $a_{i-1}=0$, for $i \in  \{2,\ldots,n-1\}$, then we omit the case $k_1=i$, since this implies that we have no intervals of length $\beta^{-i}$. 

After $m$ splittings, we have
\begin{align*}
    t_{k_1k_2\cdots k_m}^{j_1j_2 \cdots j_m} = t_{k_1k_2\cdots k_{m-1}}^{j_1j_2 \cdots j_{m-1}} + \beta^{-(k_1 + k_2 + \cdots + k_{m-1})} t_{k_m}^{j_m}.
\end{align*}
Furthermore, recall the shorthand notation
\begin{align*}
    \bk_m &\coloneqq (k_1,k_2,\ldots,k_m), \\
    \bj_m &\coloneqq (j_1,j_2,\ldots,j_m), \\
    t_{\bk_m}^{\bj_m} &\coloneqq t_{k_1k_2\cdots k_m}^{j_1j_2 \cdots j_m}, \\
    \vert \bk_m \vert & \coloneqq k_1 + k_2+ \cdots + k_m
\end{align*}
and that the points $t_{\bk_m}^{\bj_m}$ and $t_{\bk_m}^{\bj_m} + \beta^{- \vert \bk_m \vert}$ are ``consecutive points''.

\begin{lemma}\label{lem:splittingofpoints0to1}
    Let $m \geq 2$. Then for all $1\leq i\leq n$  with $a_{i-1}>0$ and $0\leq j_1\leq a_{i-1}-1$ we have 
    \begin{align}\label{fejl1}
        t_{1k_2\cdots k_m}^{\bj_m} &= \frac{j_1}{\beta} + \beta^{-1} t_{k_2\cdots k_m}^{j_2\cdots j_m}
        \quad  \text{when} \quad i=1,\end{align}
        and if $i\geq 2$
        \begin{align}\label{fejl2}
         t_{ik_2\cdots k_m}^{\bj_m} &= \frac{a_0}{\beta} + \frac{a_1}{\beta^2} + \cdots + \frac{a_{i-2}}{\beta^{i-1}} + \beta^{-(i-1)} t_{1k_2\cdots k_m}^{\bj_m}. 
    \end{align}
    
\end{lemma}
\begin{proof}
    We first prove \eqref{fejl1} by induction in $m\geq 2$. If $m=2$ then this is just the definition of $t_{1k_2}^{j_1j_2}$.  Then
    \begin{align*}
    t_{1k_2\cdots k_{m+1}}^{\bj_{m+1}} &= t_{1k_2\cdots k_{m}}^{\bj_{m}}+\beta^{-(1+k_2+\cdots + k_{m})}t_{k_{m+1}}^{j_{m+1}} \\
        &= j_1\beta^{-1} + \beta^{-1} t_{k_2\cdots k_{m}}^{j_2\cdots j_m} +\beta^{-(1+k_2+\cdots + k_{m})}t_{k_{m+1}}^{j_{m+1}} \\
        &= j_1\beta^{-1} + \beta^{-1} t_{k_2\cdots k_{m+1}}^{j_2 \cdots j_{m+1}},
        \end{align*}
        where in the second equality we used the induction step. 

        Now let us prove \eqref{fejl2}, again by induction in $m\geq 2$.  If $m=2$ we have 
        \begin{align*}t_{ik_2}^{j_1j_2}&=t_{i}^{j_1} +\beta^{-i} t_{k_2}^{j_2} =\sum_{j=0}^{i-2}\frac{a_j}{\beta^{j+1}} +\frac{j_1}{\beta^i} +\beta^{-i} t_{k_2}^{j_2}=\sum_{j=0}^{i-2}\frac{a_j}{\beta^{j+1}} +\beta^{-(i-1)}\Big (\frac{j_1}{\beta} +\beta^{-1} t_{k_2}^{j_2}\Big )\\
        &=\sum_{j=0}^{i-2}\frac{a_j}{\beta^{j+1}} +\beta^{-(i-1)}t_{1k_2}^{j_1j_2}. 
        \end{align*}
Now let $m\geq 2$. We have 
\begin{align*}
    t_{ik_2\cdots k_{m+1}}^{j_1\cdots j_{m+1}}&=t_{ik_2\cdots k_{m}}^{j_1\cdots j_{m}}+\beta^{-(i+k_2+\cdots k_m)} t_{k_{m+1}}^{j_{m+1}}=\sum_{j=0}^{i-2}\frac{a_j}{\beta^{j+1}} +\beta^{-(i-1)}t_{1k_2\cdots k_{m}}^{j_1\cdots j_{m}} +\beta^{-(i+k_2+\cdots +k_m)} t_{k_{m+1}}^{j_{m+1}}\\
    &=\sum_{j=0}^{i-2}\frac{a_j}{\beta^{j+1}} +\beta^{-(i-1)}t_{1k_2\cdots k_{m+1}}^{j_1\cdots j_{m+1}},
\end{align*}
where in the second equality we used the induction step. 
\end{proof}

We are now ready to prove Lemma~\ref{lemmacrux}.
\subsection{Proof of Lemma~\ref{lemmacrux}}
The functions appearing below were introduced in Definition \ref{Giz}. We will use induction and for the initial step we put $m=1$. 

First, let $k_1=1$, then for $0 \leq j_1 \leq a_0-1$:
\begin{align*}
    G_{s,1}^{j_1}(x) &= \frac{\beta}{\Vert B_s \Vert_{L^1([0,1])}} B_s \Big(\beta\big(x - t_1^{j_1}\big)\Big) \chi_{\big[t_1^{j_1},t_1^{j_1+1} \big]}(x) \\
    &= \frac{\beta}{\Vert B_s \Vert_{L^1([0,1])}} B_s \bigg(\beta\bigg(x - \frac{j_1}{\beta}\bigg)\bigg) \chi_{\big[j_1\beta^{-1},(j_1+1)\beta^{-1} \big]}(x).
\end{align*}
 Thus by applying the operator $\mathscr{P}$, we get for a.e. $x \in [0,1]$:
\begin{align*}
    &\beta\big(\mathscr{P}G_{s,1}^{j_1}\big)(x) = \sum_{j=0}^{a_0-1} G_{s,1}^{j_1}\bigg( \frac{x+j}{\beta}\bigg)\chi_{[0,1]}(x) +  G_{s,1}^{j_1}\bigg( \frac{x+a_0}{\beta}\bigg) \chi_{[0,1]}(x) \\
    &= \sum_{j=0}^{a_0-1} \frac{\beta}{\Vert B_s \Vert_{L^1([0,1])}} B_s \bigg(\beta\bigg(\frac{x+j}{\beta} - \frac{j_1}{\beta}\bigg)\bigg) \chi_{\big[j_1\beta^{-1},(j_1+1)\beta^{-1} \big]}\bigg( \frac{x+j}{\beta}\bigg) \chi_{[0,1]}(x)\\
    &\qquad + \frac{\beta}{\Vert B_s \Vert_{L^1([0,1])}} B_s \bigg(\beta\bigg(\frac{x+a_0}{\beta} - \frac{j_1}{\beta}\bigg)\bigg) \chi_{\big[j_1\beta^{-1},(j_1+1)\beta^{-1} \big]}\bigg( \frac{x+a_0}{\beta}\bigg) \chi_{[0,1]}(x) \\
    &= \sum_{j=0}^{a_0-1} \frac{\beta}{\Vert B_s \Vert_{L^1([0,1])}} B_s \big(x+j -j_1\big) \chi_{[j_1-j,j_1-j+1 ]}(x)\chi_{[0,1]}(x) \\
    &\qquad +  \frac{\beta}{\Vert B_s \Vert_{L^1([0,1])}} B_s \big(x+a_0 - j_1\big) \chi_{[j_1-a_0,j_1-a_0+1 ]}(x)\chi_{[0,1]}(x).
\end{align*}
Since $0 \leq j_1 \leq a_0-1$ it follows that the second term is zero a.e.. Also, the only term in the sum which is different from zero a.e. is when $j=j_1$, hence
    \begin{align}
        \big(\mathscr{P}G_{s,1}^{j_1}\big)(x) &= \frac{B_s(x)}{\Vert B_s\Vert_{L^1([0,1])}}\chi_{[0,1]}(x)= H_{sn+1}(x). \label{eq:mathscrPGnequal1}
    \end{align}
    
    Next, let $k_1=i$ for $i \in \{2,\ldots,n\}$. Notice, that this implies that $a_{i-1} \neq 0$, since otherwise we would have omitted $k_1=i$. Then for $0\leq j_1 \leq a_{i-1}-1$: 
    \begin{align*}
        G^{j_1}_{s,i} (x)&= \frac{\beta^i}{\Vert B_s \Vert_{L^1([0,1])}}B_s\Big( \beta^i \big( x - t_i^{j_1} \big) \Big)\chi_{\big[t_i^{j_1},t_i^{j_1+1} \big]}(x) \\
        &=\frac{\beta^i}{\Vert B_s \Vert_{L^1([0,1])}}B_s\bigg( \beta^i \bigg( x - \sum_{k=0}^{i-2}\frac{a_{k}}{\beta^{k+1}} -  \frac{j_1}{\beta^i}  \bigg) \bigg) \chi_{\big [\sum_{k=0}^{i-2}\frac{a_{k}}{\beta^{k+1}}+ \frac{j_1}{\beta^i},\, \sum_{k=0}^{i-2}\frac{a_{k}}{\beta^{k+1}}+ \frac{j_1+1}{\beta^i}\big ]}(x)
    \end{align*}
and 
\begin{align*}
    &\beta\big(\mathscr{P}G_{s,i}^{j_1}\big)(x) = \sum_{j=0}^{a_0-1} G_{s,i}^{j_1}\bigg( \frac{x+j}{\beta}\bigg) \chi_{[0,1]}(x)+  G_{s,i}^{j_1}\bigg( \frac{x+a_0}{\beta}\bigg) \chi_{[0,1]}(x) \\
    &= \sum_{j=0}^{a_0-1} \frac{\beta^i}{\Vert B_s \Vert_{L^1([0,1])}}B_s\bigg( \beta^i \bigg( \frac{x+j}{\beta} - \frac{a_0}{\beta} - \frac{a_1}{\beta^2} - \cdots - \frac{a_{i-2}}{\beta^{i-1}} - \frac{j_1}{\beta^i} \bigg) \bigg) \\
        &\qquad \cdot \chi_{[a_0 + \frac{a_1}{\beta} + \cdots + \frac{a_{i-2}}{\beta^{i-2}} + \frac{j_1}{\beta^{-1}}-j,a_0 + \frac{a_1}{\beta} + \cdots + \frac{a_{i-2}}{\beta^{i-2}} + \frac{j_1+1}{\beta^{i-1}}-j]}(x) \chi_{[0,1]}(x) \\
    &\qquad + \frac{\beta^i}{\Vert B_s \Vert_{L^1([0,1])}}B_s\bigg( \beta^{i-1} \bigg( x - \frac{a_1}{\beta} - \cdots - \frac{a_{i-2}}{\beta^{i-2}} - \frac{j_1}{\beta^{i-1}} \bigg) \bigg) \\
        &\qquad \cdot \chi_{[\frac{a_1}{\beta} + \cdots + \frac{a_{i-2}}{\beta^{i-2}} + \frac{j_1}{\beta^{i-1}},\frac{a_1}{\beta} + \cdots + \frac{a_{i-2}}{\beta^{i-2}} + \frac{j_1+1}{\beta^{i-1}}]}(x) \chi_{[0,1]}(x) .
\end{align*}
Since $0 \leq j \leq a_0-1$, the sum over $j$ is zero almost-everywhere and thus we get
\begin{align*}
    \Big(\mathscr{P}G_{s,i}^{j_1}\Big)(x) 
    =& \frac{\beta^{i-1}}{\Vert B_s \Vert_{L^1([0,1])}} B_s \Big( \beta^{i-1} \big(x -  \frac{a_1}{\beta} - \cdots - \frac{a_{i-2}}{\beta^{i-2}} - \frac{j_1}{\beta^{i-1}} \big)\Big)\\
    &\times\chi_{[\frac{a_1}{\beta} + \cdots + \frac{a_{i-2}}{\beta^{i-2}} + \frac{j_1}{\beta^{i-1}},\frac{a_1}{\beta} + \cdots + \frac{a_{i-2}}{\beta^{i-2}} + \frac{j_1+1}{\beta^{i-1}}]}(x) 
\end{align*}
Another $(i-1)$ applications of $\mathscr{P}$ (while noticing that each time only one term is different from zero a.e., either one in the sum if $a_\nu < a_0$ or the last term if $a_\nu = a_0$, for $\nu \in \{1,\ldots,n-1\}$) and using \eqref{eq:mathscrPGnequal1} (since $a_{i-1} \leq a_0$), gives
\begin{align*}
    \Big(\mathscr{P}^iG_{s,i}^{j_1}\Big)(x) =  H_{sn+1}(x),
\end{align*}
which finishes the basis step with $m=1$ of our induction proof. 

Now we assume that the equality is true for $m=\ell$ and let $m=\ell+1$. We first consider the case when $k_1=1$. By Lemma~\ref{lem:splittingofpoints0to1} we may write
\begin{align*}
        &G_{s,1k_2\cdots k_{\ell+1}}^{\bj_{\ell+1}} (x) = \frac{\beta^{1+k_2+\cdots + k_{\ell+1}}}{\Vert B_s \Vert_{L^1([0,1])}}B_s \Big( \beta^{1+k_2+\cdots+k_{\ell+1}} \Big(x-t_{1k_2\cdots k_{\ell+1}}^{\bj_{\ell+1}}\Big)\Big) \chi_{\big[t_{1k_2\cdots k_{\ell+1}}^{\bj_{\ell+1}},t_{1k_2 \cdots k_{\ell+1}}^{\bj_\ell,j_{\ell+1}+1}\big]}(x) \\
        &= \frac{\beta^{1+k_2+\cdots+ k_{\ell+1}}}{\Vert B_s \Vert_{L^1([0,1])}}B_s \bigg( \beta^{1+k_2+\cdots+k_{\ell+1}} \bigg(x-\frac{j_1}{\beta }-\beta^{-1}t_{k_2\cdots k_{\ell+1}}^{j_2\cdots j_{\ell+1}}\bigg)\bigg) \\
        &\qquad \cdot \chi_{\big[j_1\beta^{-1} + \beta^{-1}t_{k_2\cdots k_{\ell+1}}^{j_2 \cdots j_{\ell+1}},j_1\beta^{-1}+\beta^{-1}t_{k_2 \cdots k_{\ell+1}}^{j_2\ldots j_{\ell}(j_{\ell+1}+1)}\big]}(x),
    \end{align*}
    for $0 \leq j_1 \leq a_0-1$. Thus for a.e. $x \in [0,1]$ we have 
    \begin{align*}
        &\beta\Big( \mathscr{P}^{1 + k_2 + \cdots + k_{\ell+1}} G_{s,1k_2\cdots k_{\ell+1}}^{\bj_{\ell+1}} \Big)(x) = \Big( \mathscr{P}^{k_2 + \cdots + k_{\ell+1}} \Big( \beta \mathscr{P} G_{s,1k_2\cdots k_{\ell+1}}^{\bj_{\ell+1}} \Big)\Big)(x) \\
        &= \mathscr{P}^{k_2+\cdots +k_{\ell+1}}\bigg(\sum_{j=0}^{a_0-1} G_{s,1k_2\cdots k_{\ell+1}}^{\bj_{\ell+1}}\bigg( \frac{x+j}{\beta}\bigg)\chi_{[0,1]}(x) + G_{s,1k_2\cdots k_{\ell+1}}^{\bj_{\ell+1}}\bigg( \frac{x+a_0}{\beta}\bigg) \chi_{[0,1]}(x)\bigg) \\
        &= \Big(\mathscr{P}^{k_2+\cdots +k_{m+1}}G_{s,k_2\cdots k_{m+1}}^{j_2\cdots j_{m+1}}\Big)(x) \\
        &=H_{sn+1}(x),
    \end{align*}
    where in the third equality we have used that the second term is zero a.e., and that in the sum, only the term $j=j_1$ is not zero a.e.. Furthermore, in the last equality we have used the induction hypothesis.

    Next let $i \in \{2,\ldots,n\}$ and $k_1=i$. Recall, that this implies that $a_{i-1} \neq 0$. By Lemma~\ref{lem:splittingofpoints0to1} we may write
    \begin{align*}
        &G_{s,ik_2\cdots k_{\ell+1}}^{\bj_{\ell+1}} (x) = \frac{\beta^{i+k_2+\cdots + k_{\ell+1}}}{\Vert B_s \Vert_{L^1([0,1])}}B_s \Big( \beta^{i+k_2+\cdots+k_{\ell+1}} \Big(x-t_{ik_2\cdots k_{\ell+1}}^{\bj_{\ell+1}}\Big)\Big) \chi_{\big[t_{ik_2\cdots k_{\ell+1}}^{\bj_{\ell+1}},t_{ik_2 \cdots k_{\ell+1}}^{\bj_\ell,j_{\ell+1}+1}\big]}(x) \\
        &= \frac{\beta^{i+k_2+\cdots + k_{\ell+1}}}{\Vert B_s \Vert_{L^1([0,1])}}B_s \bigg( \beta^{i+k_2+\cdots+k_{\ell+1}} \bigg(x-\sum_{m=0}^{i-2}\frac{a_m}{\beta^{m+1}} -\beta^{-(i-1)}t_{1 k_2\cdots k_{\ell+1}}^{\bj_{\ell+1}}\bigg)\bigg) 
        \\
        &\qquad \cdot \chi_{\big[\sum_{m=0}^{i-2}\frac{a_m}{\beta^{m+1}} +{\beta^{-(i-1)}t_{1 k_2\cdots k_{\ell+1}}^{\bj_{\ell+1}}},\sum_{m=0}^{i-2}\frac{a_m}{\beta^{m+1}} + {\beta^{-(i-1)}t_{1 k_2\cdots k_{\ell+1}}^{\bj_{\ell},j_{\ell+1}+1}}\big]}(x),
        \end{align*}
        for $0 \leq j_1 \leq a_{i-1}-1$. Thus for a.e. $x \in [0,1]$ we have
        \begin{align*}
        &\beta\Big( \mathscr{P}^{i + k_2 + \cdots + k_{\ell+1}} G_{s,ik_2\cdots k_{\ell+1}}^{\bj_{\ell+1}} \Big)(x) = \Big(\mathscr{P}^{(i-1)+k_2 + \cdots + k_{\ell+1}} \Big( \beta \mathscr{P} G_{s,ik_2\cdots k_{\ell+1}}^{\bj_{\ell+1}} \Big)\Big)(x) \\
        &= \mathscr{P}^{(i-1)+k_2+\cdots +k_{\ell+1}}\bigg(\sum_{j=0}^{a_0-1} G_{s,ik_2\cdots k_{\ell+1}}^{\bj_{\ell+1}}\bigg( \frac{x+j}{\beta}\bigg)\chi_{[0,1]}(x) + G_{s,ik_2\cdots k_{\ell+1}}^{\bj_{\ell+1}}\bigg( \frac{x+a_0}{\beta}\bigg) \chi_{[0,1]}(x)\bigg) \\
        &= \mathscr{P}^{(i-1)+k_2+\cdots +k_{m+1}}  \bigg( \frac{\beta^{i+k_2+\cdots + k_{\ell+1}}}{\Vert B_s \Vert_{L^1([0,1])}}B_s \bigg( \beta^{i+k_2+\cdots+k_{\ell+1}} \bigg(\frac{x}{\beta}- \sum_{m=1}^{i-2}\frac{a_{m}}{\beta^{m+1}}-{ \beta^{-(i-1)}t_{1k_2\cdots k_{\ell+1}}^{\bj_{\ell+1}}}\bigg)\bigg) \\
        &\qquad \cdot \chi_{\big[\frac{a_0}{\beta} + \cdots + \frac{a_{i-2}}{\beta^{i-1}}+{\beta^{-(i-1)}t_{1k_2\cdots k_{\ell+1}}^{\bj_{\ell+1}}},\frac{a_0}{\beta} + \cdots + \frac{a_{i-2}}{\beta^{i-1}}+{ \beta^{-(i-1)}t_{1k_2\cdots k_{\ell+1}}^{\bj_{\ell},j_{\ell+1}+1}}\big]}\bigg(\frac{x+a_0}{\beta}\bigg)\bigg) \\
        &= \mathscr{P}^{(i-1)+k_2+\cdots +k_{m+1}}  \bigg( \frac{\beta^{i+k_2+\cdots + k_{\ell+1}}}{\Vert B_s \Vert_{L^1([0,1])}}B_s \bigg( \beta^{(i-1)+k_2+\cdots+k_{\ell+1}} \bigg(x - \sum_{m=1}^{i-2}\frac{a_m}{\beta^{m}}-{\beta^{-(i-2)}t_{1k_2\cdots k_{\ell+1}}^{\bj_{\ell+1}}}\bigg)\bigg) \\
        &\qquad \cdot \chi_{\big[\sum_{m=1}^{i-2}\frac{a_m}{\beta^{m}} +{\beta^{-(i-2)}t_{1 k_2\cdots k_{\ell+1}}^{\bj_{\ell+1}}},\sum_{m=1}^{i-2}\frac{a_m}{\beta^{m}}+{ \beta^{-(i-2)}t_{1 k_2\cdots k_{\ell+1}}^{\bj_{\ell},j_{\ell+1}+1}}\big]}(x)\bigg),
    \end{align*}
    where in the third equality, we have used that the sum over $0\leq j\leq a_0-1$ equals zero a.e.. Applying $\mathscr{P}$ another $(i-2)$ times (where each time only one term is different from zero a.e., either one in the sum if $a_\nu < a_0$ or the last term if $a_\nu = a_0$, for $\nu\in \{1,\ldots,n-1\}$) we get back to $k_1=1$, which we have already considered.
\qed 

\subsection{Proof of Lemma \ref{lemma-relabeling}}\label{sec-relab}
{
By definition, there exists $j_1<a_{{k_1}-1}$ such that 
\begin{align*}
    t_{k_1k_2\cdots k_m}^{j_1j_2\cdots j_m}=a_0/\beta +a_1/\beta^2+\cdots + a_{k_1-2}/\beta^{k_1-1} + j_1/\beta^{k_1} +\beta^{-{k_1}}t_{k_2\cdots k_m}^{j_2\cdots j_m}. 
\end{align*}
If ${k_1}=2$ then 
\begin{align*}
    t_{2k_2\cdots k_m}^{j_1j_2\cdots j_m}=a_0/\beta +j_1/\beta^2+\beta^{-2}t_{k_2\cdots k_m}^{j_2\cdots j_m}  
\end{align*}
hence $\tilde{t}=t_{1k_2\cdots k_m}^{j_1j_2\cdots j_m}$. Thus we may assume that ${k_1}\geq 3$. 

If $a_1,\ldots ,a_{{k_1}-2}<a_0$ then 
\begin{align*}
    \tilde{t}=a_1/\beta+\cdots + a_{{k_1}-2}/\beta^{{k_1}-2} + j_1/\beta^{{k_1}-1} +\beta^{-({k_1}-1)}t_{k_2\cdots k_m}^{j_2\cdots j_m}=t_{1\times ({k_1}-1) k_2\cdots k_m}^{a_1\cdots a_{{k_1}-2}j_1\cdots j_m} 
\end{align*}
where again the notation $1\times ({k_1}-1)$ means $1$ repeated $k_1-1$ times. 
Now assume that there exists $1\leq r\leq {k_1}-2$ such that $a_r=a_0$. Let $r_1$ be the smallest such index. This means that $a_1,a_2,\ldots,a_{r_1-1}<a_0$ and 
\begin{equation*}
    a_1/\beta +\cdots +a_{r_1-1}/\beta^{r_1-1}=t_{1\times(r_1-1)}^{a_1\cdots a_{r_1-1}}.
\end{equation*}
Then with $a_{r_1}=a_0$ we have 

\begin{align*}
    \tilde{t}&=t_{1\times(r_1-1)}^{a_1\cdots a_{r_1-1}} +\beta^{-(r_1-1)} \big (a_{0}/\beta +a_{r_1+1}/\beta^2+\cdots  +a_{{k_1}-2}/\beta^{{k_1}-r_1-1}+ j_1/\beta^{{k_1}-r_1}\big )\\
    &\quad +\beta^{-({k_1}-1)}t_{k_2\cdots k_m}^{j_2\cdots j_m}.
\end{align*}
We need to control the parenthesis. Parry's lexicographic condition implies: 
\begin{equation*}
    [a_{r_1},a_{r_1+1},\ldots ,a_{{k_1}-2}, j_1,0,\ldots ]\prec [a_0,a_1,\ldots, a_{n-1},0,\ldots ]. 
\end{equation*}
In the ``shortest case scenario", we have a ``best saturation" of the above inequality where $a_{r_1+1}=a_1$, $a_{r_1+2}=a_2$, ..., $a_{{k_1}-2}=a_{{k_1}-r_1-2}$ and $j_1<a_{{k_1}-1}\leq a_{{k_1}-r_1-1}$ so 
\begin{equation*}
    a_{0}/\beta +a_{r_1+1}/\beta^2+\cdots  +a_{s-2}/\beta^{k_1-r_1-1}+ j_1/\beta^{k_1-r_1}=t_{k_1-r_1}^{j_1}
\end{equation*}
and 
\begin{equation*}
    \tilde{t}=t_{1\times(r_1-1)}^{a_1\cdots a_{r_1-1}} +\beta^{-(r_1-1)}\Big (t_{{k_1}-r_1}^{j_1} +\beta^{-({k_1}-r_1)}t_{k_2\cdots k_m}^{j_2\cdots j_m}\Big )=t_{1\times(r_1-1),{k_1}-r_1,k_2\cdots k_m}^{a_1\cdots a_{r_1-1}j_1j_2\cdots j_m}.
\end{equation*}

If the best saturation does not happen, for example when $a_{r_1+1}<a_1$, we have 
\begin{equation*}
    a_{0}/\beta +a_{r_1+1}/\beta^2+\cdots  +a_{{k_1}-2}/\beta^{{k_1}-r_1-1}+ j_1/\beta^{{k_1}-r_1}=t_{2}^{a_{r_1+1}}+\beta^{-2}\big ( a_{r_1+2}/\beta +\cdots + j_1/\beta^{{k_1}-r_1-2}\big )
\end{equation*}
and we continue by induction using Parry's condition until the parenthesis is exhausted. 

Now let us prove the other way around, that is if $a_0/\beta +\beta^{-1} t_{\bk_m}^{\bj_m}<1$ then it must equal some $t_{\tilde{\bk}_p}^{\tilde{\bj}_p}$. Let us start with $k_1=1$, i.e.  
\begin{equation*}
    a_0/\beta +\beta^{-1} t_{1k_2\cdots k_m}^{j_1j_2\cdots j_m}=a_0/\beta +j_1/\beta^2 +\beta^{-2} t_{k_2\cdots k_m}^{j_2\cdots j_m}.
\end{equation*}
If $j_1>a_1$ then this number is larger than $1$. If $j_1<a_1$ then we can write it as $t_{2k_2\cdots k_m}^{j_1j_2\cdots j_m}$. If $j_1=a_1$ then we can write it as $t_{3k_2\cdots k_m}^{0j_2\cdots j_m}$. 

If $k_1=2$ then 
\begin{equation*}
    a_0/\beta +\beta^{-1} t_{2k_2\cdots k_m}^{j_1j_2\cdots j_m}=a_0/\beta +a_0/\beta^2 +j_1/\beta^3 +\beta^{-3} t_{k_2\cdots k_m}^{j_2\cdots j_m},\quad j_1<a_1.
\end{equation*}
First of all, if $a_1<a_0$ this number is larger than $1$. If $a_1=a_0$ and $j_1<a_2$ then this number can be written as $t_{3k_2\cdots k_m}^{j_1j_2\cdots j_m}$. If $j_1>a_2$ the number is larger than $1$, and if $j_1=a_2$ then it can be written as $t_{4k_2\cdots k_m}^{0j_2\cdots j_m}$. 
When $k_1\geq 3$ the argument is similar.

\qed }

\section{Euler-Bernoulli approximation formula}\label{sec:eulerbernoulliformula}
We refer the reader to \cite{CS} for proofs of the results in this section. Define ${B}_0(x)=1$ and ${B}_1(x)=x-1/2$ for $x\in [0,1)$ and $\Z$-periodically extend both $B_0$ and $B_1$ to all of $\mathbb{R}$; the periodic functions are denoted by $\widetilde{B}_0$ and $\widetilde{B}_1$. Let $r\geq 2$ and define 
\begin{align*}
\widetilde{B}_r(x)\coloneqq -r!\sum_{m\neq 0} e^{2\pi i m x}\, (2\pi i m)^{-r},\quad \frac{d\widetilde{B}_r}{dx}(x)= r\widetilde{B}_{r-1}(x),\quad 0<x<1.
\end{align*}
Then $\widetilde{B}_r$ are $\mathbb{Z}$-periodic functions which coincide with the Bernoulli polynomials $B_r$ on $(0,1)$ for $r \in \mathbb{N}$. If $\widetilde{F}\in C^N([0,1])$ with $N\geq 1$, the Euler-Bernoulli approximation formula reads as:
\begin{equation}\label{cdh22}
    \begin{aligned}
    \widetilde{F}(x)=   \int_0^1 \widetilde{F}(t) \d t &+\sum_{s=1}^N \big ( \widetilde{F}^{(s-1)}(1)-\widetilde{F}^{(s-1)}(0)\big ) \frac{\widetilde{B}_s(x)}{s!}  \\
     &-\int_{0}^{1}\, \widetilde{F}^{(N)}(u)\, \frac{\widetilde{B}_N(x-u)}{N!} \d u,\quad 0< x< 1.  
    \end{aligned}
\end{equation}
This can be rescaled to general intervals. If $F \colon [a,b]\to \R$ belongs to  $C^N([a,b])$ we can define $\widetilde{F}(x)=F\big (a(1-x)+bx\big )$ with $0\leq x\leq 1$. Then for all $a<y<b$ we have
\begin{equation}\label{BAF}
    \begin{aligned}
     F(y)&=   (b-a)^{-1}\int_a^b F(u) \d u +\sum_{s=1}^N (b-a)^{s-1}\big ( F^{(s-1)}(b)-F^{(s-1)}(a)\big ) \frac{\widetilde{B}_s\big( \frac{y-a}{b-a}\big)}{s!} \\
     &\qquad -(b-a)^N\int_{0}^{1}\, F^{(N)}\big (a(1-u)+bu\big )\, \frac{\widetilde{B}_N\big (\frac{y-a}{b-a}-u \big )}{N!} \d u. 
    \end{aligned}
\end{equation}
By this rescaled Bernoulli approximation formula applied to each interval in the splitting, i.e. $b=t_{\bk_m}^{\bj_m} + \beta^{- \vert \bk_m \vert}$ and $a=t_{\bk_m}^{\bj_m}$ we get
\begin{equation}\label{aug22}
    \begin{aligned}
     F(y)&=  \beta^{|\bk_m|} \int_{t_{\bk_m}^{\bj_m} }^{t_{\bk_m}^{\bj_m} + \beta^{- \vert \bk_m \vert} } F(u) \d u \\
     &\qquad +\sum_{s=1}^N \beta^{-(s-1)|\bk_m|}\big ( F^{(s-1)}(t_{\bk_m}^{\bj_m} + \beta^{- \vert \bk_m \vert} )-F^{(s-1)}(t_{\bk_m}^{\bj_m} )\big ) \frac{\widetilde{B}_s\Big( \beta^{|\bk_m|}\big (y-t_{\bk_m}^{\bj_m}\big )\Big)}{s!} \\
     &\qquad -\beta^{-N|\bk_m|}\int_{0}^{1}\, F^{(N)}\big (t_{\bk_m}^{\bj_m} (1-u)+t_{\bk_m}^{\bj_m} + \beta^{- \vert \bk_m \vert} u\big )\, \frac{\widetilde{B}_N\Big( \beta^{|\bk_m|}\big (y-t_{\bk_m}^{\bj_m}\big )-u\Big)}{N!} \d u \\
     &=   \beta^{|\bk_m|} \int_{t_{\bk_m}^{\bj_m} }^{t_{\bk_m}^{\bj_m} + \beta^{- \vert \bk_m \vert} } F(u) \d u  \\
     &\qquad +\sum_{s=1}^N \beta^{-(s-1)|\bk_m|}\big ( F^{(s-1)}(t_{\bk_m}^{\bj_m} + \beta^{- \vert \bk_m \vert} )-F^{(s-1)}(t_{\bk_m}^{\bj_m} )\big ) \frac{\widetilde{B}_s\Big( \beta^{|\bk_m|}\big (y-t_{\bk_m}^{\bj_m}\big )\Big)}{s!}\\
     &\qquad +\mathcal{O}\Big( \beta^{-|\bk_m|N}\Vert F^{(N)}\Vert_{C([0,1])}\Big),\quad \forall y\in \big [t_{\bk_m}^{\bj_m}\, ,\, t_{\bk_m}^{\bj_m} + \beta^{- \vert \bk_m \vert}\big ].
    \end{aligned}
    \end{equation} 
        
    \section{Proofs of lemmas regarding bounded variation}
    \subsection{Proof of Lemma \ref{lem1}}\label{sec-lem1}

    Let us denote by $X_1=\W f -g_1$. Consider a partition $0=x_0<x_1<\cdots <x_m=1$ as in \eqref{beve}. There exists some $1\leq j_0\leq m-1$ such that $x_{j_0}\leq b_1<x_{j_0+1}$ which leads to  
\begin{align*}
\sum_{j=0}^{m-1}\big \vert X_1(x_{j+1})-X_1(x_j)\vert &=\sum_{j=0}^{j_0-1}\big \vert (\W f) (x_{j+1})-(\W f)(x_j)\vert \\
&\quad + |X_1(x_{j_0+1})-X_1(x_{j_0})|+\sum_{j=j_0+1}^{m-1}\big \vert (\W f) (x_{j+1})-(\W f)(x_j)\vert .
\end{align*}

We have $f\big ((a_0+x_{j_0+1})/\beta\big )=g_1\big (x_{j_0+1}\big )=0$ and $g_1\big (x_{j_0}\big )=f(1)\beta^{-1}$ hence 
\begin{align*}
X_1(x_{j_0+1})-X_1(x_{j_0})&=(\W f)(x_{j_0+1})-(\W f)(x_{j_0})+f(1)\beta^{-1}\\
&=\beta^{-1}\sum_{k=0}^{a_0-1} \Big (f((k+x_{j_0+1})/\beta)-f((k+x_{j_0})/\beta)\Big )\\
&\quad +\beta^{-1}\Big (f(1)-f((a_0+x_{j_0})/\beta)\Big ).
\end{align*}

In the sum over $j>j_0$, the term with $k=a_0$ is absent, thus adding all three contributions we have
\begin{align*}
\sum_{j=0}^{m-1}\big \vert X_1(x_{j+1})-X_1(x_j)\big \vert &\leq \beta^{-1} \sum_{j=0}^{j_0-1}\sum_{k=0}^{a_0} \big \vert  f((k+x_{j+1})/\beta)-f((k+x_{j})/\beta)\big \vert \\
&\quad +\beta^{-1} \sum_{k=0}^{a_0-1}\big \vert f\big ((k+x_{j_0+1})/\beta\big )-f\big ((k+x_{j_0})/\beta\big )  \\
&\quad +\beta^{-1}   |f(1)-f((a_0+x_{j_0})/\beta)| \\
&\quad + \beta^{-1}\sum_{j=j_0+1}^{m-1}
\sum_{k=0}^{a_0-1} \big \vert  f((k+x_{j+1})/\beta)-f((k+x_{j})/\beta)\big \vert. 
\end{align*}

The points $y_{k,j}:=(k+x_j)/\beta$ live in the intervals $[k/\beta,(k+1)/\beta]$ and form a partition of $[0,(a_0+1)/\beta]$. We see that the largest element of this partition for which $f(y_{k,j})$ might be non zero is $(a_0+x_{j_0})/\beta$. Rearranging the sums by first summing over $k\leq a_0-1$ and after that for $k=a_0$, the right hand side of the above inequality becomes: 
\begin{align*}
& \beta^{-1} \sum_{k=0}^{a_0-1} \sum_{j=0}^{m-1}\big \vert  f((k+x_{j+1})/\beta)-f((k+x_{j})/\beta)\big \vert \\
&\quad +\beta^{-1} \sum_{j=0}^{j_0-1}\big \vert f\big ((a_0+x_{j+1})/\beta\big )-f\big ((a_0+x_{j})/\beta\big )  +\beta^{-1}   |f(1)-f((a_0+x_{j_0})/\beta)|\\
&=\beta^{-1}\big(\sum_{k=0}^{a_0-1}\sum_{j=0}^{m-1}|f(y_{k,j+1})-f(y_{k,j})| + \sum_{j=0}^{j_0-1}|f(y_{a_0,j+1})-f(y_{a_0,j})|+ |f(1)-f(y_{a_0,j_0})|\big)\\
&\leq \beta^{-1} {\rm Var}_0^1(f). 
\end{align*} \qed 

\subsection{Proof of Lemma \ref{lemmap2}} \label{sec-lemmap2}
Let $g_1(x)=f(1)\beta^{-1} \big (-b_1 +\chi_{[0,b_1]}(x)\big )$. It only differs by a global constant from the function $g_1$ from Lemma \ref{lem1}, while here $\int_0^1 g_1(x) \d x=0$. We still have  
    ${\rm Var}_0^1(\W f -g_1)\leq \beta^{-1}{\rm Var}_0^1(f)$. For $r\geq 1$ we introduce 
    \begin{equation}\label{jun1'}
    g_{r+1}(x):=(\W g_r)(x)+\big ((\W^rf)(1)-g_r(1)\big ) \beta^{-1} 
    \big (-b_1 +\chi_{[0,b_1]}(x)\big ).
    \end{equation}
    We have 
    \begin{equation*}
        \W^{r+1}f-g_{r+1}=\W (\W^rf-g_r) -\big ((\W^rf)(1)-g_r(1)\big ) \beta^{-1} \big (-b_1 +\chi_{[0,b_1]}\big )
    \end{equation*}
    hence the induction step implies 
    \begin{equation*}
        {\rm Var}_0^1 (\W^{r+1}f-g_{r+1})\leq \beta^{-1} {\rm Var}_0^1 (\W^{r}f-g_{r})\leq \beta^{-(r+1)}{\rm Var}_0^1(f).
    \end{equation*}
    In conclusion, we have just proved that the function $X_r=\W^r f-g_r$ obeys for all $r\geq 1$ that 
    \begin{equation}\label{jun2}
        {\rm Var}_0^1(X_r)\leq \beta^{-r} {\rm Var}_0^1(f)\quad \text{and}\quad  \int_0^1 X_r(x) \d x=0, 
    \end{equation}
where the last equality comes from the fact that $\W$ preserves the value of integrals. Since $\int_0^1 X_r(x) \d x=0$ then $X_r$ must change sign, thus $|X_r(x)|\leq {\rm Var}_0^1(X_r)\leq \beta^{-r}{\rm Var}_0^1(f)$ for all $x\in [0,1]$. 
    In particular, if we denote by 
    $c_r=(\W^rf)(1)-g_r(1)$ we have 
    \begin{equation}\label{jun3}
        |c_r|\leq \beta^{-r}{\rm Var}_0^1(f),\quad r\geq 1.
    \end{equation}
Let $h:=\beta^{-1} (-b_1+\chi_{[0,b_1]})$.  We can rewrite \eqref{jun1'} as 
\begin{equation*}
    g_{r+1}=\W g_r +c_r h,\quad r\geq 1.
\end{equation*}
Then for all $r\geq 2$ we have 
\begin{equation}\label{gem}
\begin{aligned} 
g_r&=\W g_{r-1} +c_{r-1} h= \W^2 g_{r-2} +c_{r-2}\W h +c_{r-1} h\\
&=\W^{r-1} g_1 +\sum_{k=1}^{r-1}c_{r-k}\W^{k-1}h.
\end{aligned}
\end{equation}
We observe that $h$ and all the functions $g_r$ are linear combinations of the functions $\chi_{[0,b_j]}$ with $0\leq j\leq n-1$, which form an invariant subspace for $\W$, see Appendix \ref{sec:invsubspace}. The restriction of $\beta \W$ to this $n$ dimensional subspace generates the matrix from \eqref{Pn}. The spectral properties of this matrix are analyzed in Lemma \ref{lemmahc1}.

Since $\int_0^1 g_1(x) \d x=\int_0^1 h(x) \d x=0$, both $g_1$ and $h$ must have a zero component on the subspace corresponding to the eigenvalue $\beta$ of the matrix $\mathcal{P}_n$ from \eqref{Pn}.  This can be seen as follows. Let for example $h(x)=c_1 u_1(x)+\tilde{h}(x)$ where $\tilde{h}(x)$ belongs to the spectral subspaces of all other eigenvalues of $\mathcal{P}_n$. Then $\W^m h-c_1 u_1 =\W^m \tilde{h}\to 0$ when $m\to\infty$. Since $\W$ preserves the integral of $h$, we must have $c_1=0$. 

Let $\delta=1$ if the algebraic  multiplicity of $\lambda_*$ is larger than $1$ (but at most $n-1$), and $\delta=0$ if $\lambda_*$ is simple. Then there exists $C>0$ such that for all $j\geq 1$ we have
\begin{equation}\label{juli1}
\max\big \{ \Vert \W^{j}g_1\Vert_{BV}\, ,\,  \Vert \W^{j}h\Vert_{BV}\big \} \leq   C\, j^{\delta (n-1)} (|\lambda_*|/\beta)^j.
\end{equation} 
This together with \eqref{jun3} applied to \eqref{gem} leads to (remember that $|f(1)|\leq \Vert f\Vert_{L^\infty}\leq {\rm Var}_0^1(f)$ because $f$ changes sign):
\begin{equation}\label{jun4}
\Vert g_r-(\beta^{-1}\mathcal{P}_n)^{r-1} g_1\Vert_{BV}\leq  C {\rm Var}_0^1(f)\sum_{k=1}^{r-1}\beta^{-(r-k)}  k^{\delta (n-1)} (|\lambda_*|/\beta)^{k}.
\end{equation}
We have the following alternatives: 
\begin{itemize}
    \item If $|\lambda_*|< 1$ then  we know from Lemma \ref{simple-eig} that all roots of $p(z)$ are simple, thus $\delta=0$ in \eqref{juli1} and \eqref{jun4} hence $\sum_{k=1}^{r-1}\beta^{-(r-k)}  (|\lambda_*|/\beta)^{k}\leq \beta^{-r}(1-|\lambda_*|)^{-1}$ and 
    \begin{equation*}
        \Vert g_r-(\beta^{-1}\mathcal{P}_n)^{r-1} g_1\Vert_{BV}\leq C {\rm Var}_0^1(f)\, (1-|\lambda_*|)^{-1}\beta^{-r};
    \end{equation*}
    \item If $|\lambda_*|=1$ then $\Vert g_r-(\beta^{-1}\mathcal{P}_n)^{r-1} g_1\Vert_{BV}\leq C {\rm Var}_0^1(f)\, \beta^{-r}\,  r^n.$
    \item If $1< |\lambda_*|< \beta$ then $\sum_{k=1}^{r-1}\beta^{-(r-k)}  k^{n-1} (|\lambda_*|/\beta)^{k}= r^{n-1} \beta^{-r}\sum_{k=1}^{r-1}|\lambda_*|^k\leq \frac{ r^{n-1}}{|\lambda_*|-1}(|\lambda_*|/\beta)^r$ and 
    \begin{equation*}
        \Vert g_r-(\beta^{-1}\mathcal{P}_n)^{r-1} g_1\Vert_{BV}\leq C {\rm Var}_0^1(f)\,\frac{ r^{n-1}(|\lambda_*|/\beta)^{r}}{|\lambda_*|-1} .
    \end{equation*}
\end{itemize}

Going back to \eqref{jun2} and using that $X_r=\W^r f-g_r$ changes sign we have 
\begin{equation*}
    \Vert \W^r f-g_r\Vert_{L^\infty}\leq {\rm Var}_0^1(\W^r f-g_r)\leq {\rm Var}_0^1(f)\, \beta^{-r}
\end{equation*}
which together with the above estimates on \eqref{jun4} and the triangle inequality we obtain \eqref{jun1}.  \qed
\subsection{Proof of Lemma \ref{lemmap1}}\label{sec-lemmap1}

Let $g=f-\int_0^1f(x)\d x$. Then $\int_0^1g(x)\d x=0$. We will now reuse some of the construction from Lemma \ref{lemmap2}. Namely, from \eqref{jun1'} we see that all $g_r$'s with $r\geq 1$ belong to the subspace generated by the $\chi_{[0,b_j]}$'s.  From \eqref{jun2}  we have ${\rm Var}_0^1(X_r)\leq \beta^{-r} {\rm Var}_0^1(g)$. Thus by putting $h_r=g_r+\chi_{[0,1]} \int_0^1 f(x)\d x $ we obtain the result.

\section{Various properties of $\beta$}

\subsection{Criteria for $p(z)$ to be a Parry-Pisot  polynomial}

We start by giving here a short alternative proof for a result of Brauer \cite{Br} which holds for all $n\geq 2$.  
\begin{lemma}\label{brauer}Assume that $n\geq 2$ and $a_0\geq \cdots \geq a_{n-1}\geq 1$ are natural numbers. Let 
\begin{equation*}
    p(z)=z^n-(a_0z^{n-1}+\cdots +a_{n-1}).
\end{equation*} 
Then $p(z)$ is a Parry-Pisot polynomial according to Definition \ref{def1}.  
\end{lemma}
\begin{proof}
    We already know from the analysis of \eqref{dec2} that $f(x)=x^{-n}p(x)$ (and thus $p$) has a unique and simple positive root $\beta>1$. Also, from \eqref{Parrycon}  
    we infer  that $p(z)$ is a Parry polynomial. We now need to show that the other roots of $p(z)$ belong to the open unit disk.
    
    Let 
    \begin{equation}\label{dec20}
    \begin{aligned}q(z)&=(1-z)p(z)=(1-z)\big (z^n-( a_0z^{n-1}+\cdots+ a_{n-1})\big )\\
    &=(a_0+1)z^n -\Big (z^{n+1}+(a_0-a_1)z^{n-1}+\cdots +(a_{n-2}-a_{n-1})z+a_{n-1}\Big )\\
    &=: H(z)-M(z).
    \end{aligned}
    \end{equation}
    Let us show that if $|z|=1+r$ with $r>0$ sufficiently small then $|H(z)|>|M(z)|$, hence $q$ has no zeros on that circle. Indeed, if $r$ is small we have 
    \begin{equation*}
        \vert H(z) \vert=(a_0+1)(1+nr)+\mathcal{O}(r^2)
    \end{equation*}
    and by the triangle inequality
\begin{align*}
&\vert M(z) \vert \leq 1+(n+1)r+(a_0-a_1)(1+(n-1)r)+\cdots +(a_{n-2}-a_{n-1})(1+r) +a_{n-1}+\mathcal{O}(r^2)\\
&=1+a_0 +r\Big (n+1 +(a_0-a_1)(n-1)+(a_1-a_2)(n-2)+\cdots +(a_{n-2}-a_{n-1})\Big )+\mathcal{O}(r^2).
\end{align*}
Replacing $n-k$ with $n-1$ we have 
\begin{align*}(a_0-a_1)(n-1)+(a_1-a_2)(n-2)\cdots +(a_{n-2}-a_{n-1})&\leq (n-1)(a_0-a_{n-1})
\end{align*}
and thus
\begin{align*}|M(z)|&\leq 1+a_0 +r\big (n+1 +(a_0-a_{n-1})(n-1)\big )+\mathcal{O}(r^2).
\end{align*}
Hence for small $r>0$:

\begin{align*}r^{-1}\big (|H(z)|-|M(z)|\big )&\geq  (1+a_0)n -(n+1) -(a_0-a_{n-1})(n-1)+\mathcal{O}(r)\\
& = a_0-1 +a_{n-1}(n-1)+\mathcal{O}(r) \geq  1+\mathcal{O}(r)>0.
\end{align*}
Now fix such an $r$ and consider the polynomial 
\begin{equation*}
    q_t(z)=H(z)-tM(z),\quad \vert z \vert \leq 1+r,\quad 0\leq t\leq 1,\quad q(z)=q_1(z),\quad H(z)=q_0(z).
\end{equation*}
As in Rouch{\'e}'s theorem, the number of zeros (counting algebraic multiplicities) of $q_t$ inside the disk $|z|\leq 1+r$ is given by the Cauchy integral 
\begin{equation*}
    N_t=\frac{1}{2\pi i}\oint_{|z|=1+r}\frac{q_t'(z)}{q_t(z)}\, dz
\end{equation*}
which is integer valued and (Lipschitz) continuous on $0\leq t\leq 1$, thus constant and equal to $N_0=n$. This remains true for all small enough $r$, thus $q(z)=q_1(z)$ has exactly $n$ zeros inside the closed unit disk $\vert z \vert \leq 1$. 

By construction, $z=1$ is a root for $q$ but not for $p$. Let us show that $q$ has no other roots on $\vert z \vert=1$ other than $z=1$. Indeed, if $z=e^{i\theta}$ is such a root with $-\pi<\theta\leq \pi$ then $z^{-n}q(z)=0$ and from \eqref{dec20} we get
\begin{align*}a_0+1&=\cos(\theta)+(a_0-a_1)\cos(\theta)+\cdots + (a_{n-2}-a_{n-1})\cos((n-1)\theta)+a_{n-1}\cos(n\theta)\\
&\leq \cos(\theta)+a_0
\end{align*}
which demands $\theta=0$. We conclude that $q$ thus $p$ have $n-1$ zeros inside the open unit disk.  
    \end{proof}

    \begin{lemma}\label{brauer'}Assume that $\{a_j\}_{j=0}^{n-1}$ with $n\geq 2$ are natural numbers such that 
    \begin{equation*}
        a_{n-1}\geq 1 \quad \text{and}\quad a_0\geq 2+a_1+\cdots +a_{n-1}.
    \end{equation*}
Then they form a finite Parry sequence  and  $p(z)=z^n-(a_0z^{n-1}+\cdots +a_{n-1})$ 
is Parry-Pisot according to Definition \ref{def1}. 
\end{lemma}
\begin{proof}
The sequence $\{a_j\}_{j=0}^{n-1}$ obeys \eqref{dec5} with $j=0$ hence $p(z)$ is a Parry polynomial. We write $p(z)=H(z)+M(z)$ where \begin{equation*}
    H(z)=-a_0z^{n-1} \quad \text{and}\quad  M(z)=z^n -a_1z^{n-2}-\cdots -a_{n-1}.
\end{equation*} 
If $|z|=1$ then 
$\vert H(z) \vert - \vert M(z) \vert\geq a_0-(1+a_1+\cdots+a_{n-1})\geq 1$ hence reasoning as in the previous Lemma we have that $p(z)$ and $H(z)$ have exactly $n-1$ zeros inside the open unit disk while $\beta>1$ is also a zero for $p(z)$, hence $\beta$ is a Pisot number. 
\end{proof}

    It turns out that if $n\leq 3$, then $p(z)$ is always Parry-Pisot: 
    \begin{lemma}\label{pisot}
For any finite Parry sequence with $n\leq 3$, $p(z)$ is a Parry-Pisot polynomial.
\end{lemma}
\begin{proof}
If $n=2$ then we must have $a_0\geq a_1$ hence Brauer's criterion is always satisfied, but one can also directly use the formula which gives the second root $z_2=-a_1/\beta$ which shows that $|z_2|<1$. If $n=3$ we might have that $a_0\geq \max\{a_1,a_2\}$ and $a_1<a_2$, which still generate a finite Parry sequence according to Lemma \ref{a0islarger} but do not satisfy Brauer's criterion.  Nevertheless, let us show that 
\begin{equation*}
    p(z)=z^3-a_0 z^2-a_1 z-a_2
\end{equation*}
is Parry-Pisot. 
Denote its roots with $\beta>1$,  $z_2$ and $z_3$. If $z_2$ has a non-zero imaginary part then $z_3=\overline{z_2}$. We must have $\beta z_2z_3=a_2$, hence $|z_2|=|z_3|=\sqrt{a_2/\beta}\leq \sqrt{a_0/\beta}< 1$. If $z_2$ and $z_3$ are both real, then they must be negative since $\beta$ is the only positive root. We also have 
\begin{equation*}
    \beta+z_2+z_3 =a_0
\end{equation*}
hence $0<-z_2-z_3=\beta-a_0<1$ which implies that $z_{2,3}\in (-1,0)$. We note that in the ``non Brauer" case with $a_1<a_2$ we cannot have negative roots at all because for $z\in [-1,0]$ we have 
\begin{equation*}
    p(z)<-a_1 z-a_2\leq a_1-a_2<0.
\end{equation*}
\end{proof}

\begin{lemma}\label{simple-eig}
If $p(z)$ is Parry-Pisot, then it is irreducible over the rationals and all its roots are simple.
\end{lemma}
\begin{proof}
   Let us first prove that $p(z)$ cannot be written as a product of two nontrivial monic polynomials with integer coefficients. If this {were} possible, then one of the factors { would have} $\beta$ as a root, while the other factor can only have roots of absolute value obeying $0<|\lambda|<1$ hence its ``free" term equals (up to a sign)  the product of its roots, which cannot equal an integer.   

From Gauss' Lemma we have that $p(z)$ is also irreducible over the rationals. Denote by $d(z)$ the greatest common divisor of $p(z)$ and $p'(z)$. From B{\'e}zout's identity we know that $d(z)$ can be written as a linear combination of $p(z)$ and $p'(z)$, hence if $p(z)$ had a double root it would also be a root for $d(z)$, implying that the degree of $d(z)$ is at least $1$. This contradicts the irreducibility of $p(z)$ over the rationals. 
\end{proof}

\subsection{All roots of $p(z)$ can get arbitrarily close to the unit circle}\label{appE2}
The interest of this result is more that there exist Parry polynomials whose roots can get  arbitrarily close to 1 in absolute value, that is given $\epsilon$ there is a Parry polynomial all of whose roots $z$ obey $|1-|z|| < \epsilon $ .
Let us consider the polynomial $p(z)=z^n-z^{n-1}-1$ with $n\geq 3$ for which $a_0=a_{n-1}=1$ and $a_1=\cdots =a_{n-2}=0$.  

Let $f(z)=-p(z)/z^{n-1}=1+1/z^{n-1} -z$. The zeros of $f$ coincide with the zeros of $p$.  Let us estimate the position of $\beta$. We have $f(1)>0$ and 
\begin{equation*}
    f\big (1+1/\sqrt{n-1}\big )=e^{-(n-1)\ln (1+1/\sqrt{n-1})}-1/\sqrt{n-1}=-1/\sqrt{n-1} + \mathcal{O}(e^{-\sqrt{n-1}})<0
\end{equation*}
if $n$ is large enough. Thus $\beta\in (1,1+1/\sqrt{n-1})$ for $n$ large enough.

Since $a_0=1$  we have $\beta<a_0+1<2$. Thus 
if $z$ is any zero of $p(z)$ we have  $|z|\leq \beta \leq 2$ and 
\begin{equation*}
    \vert 1-z\vert =1/\vert z\vert ^{n-1}\leq 3,\quad \vert z\vert \geq 1/3^{1/(n-1)}.
\end{equation*}
This means that for large enough $n$, the zeros of $p$ obey 
\begin{equation*}
    1/3^{1/(n-1)}\leq \vert z\vert \leq 1+1/\sqrt{n-1},
\end{equation*}
thus  $\vert z \vert $ can get arbitrarily close to 1 when $n$ grows. 
\end{appendix}

\begin{Backmatter}

\begin{ack}
We are grateful for the technical assistance of A. Author.
\end{ack}

\end{Backmatter}

\end{document}